\documentclass[11pt]{article}

\usepackage[english]{babel}
\usepackage[utf8x]{inputenc}
\usepackage[T1]{fontenc} 
\usepackage{scalerel}
\usepackage[dvipsnames]{xcolor}
\usepackage[margin=1.1in]{geometry}

\usepackage[math]{kurier}
\usepackage[sc]{mathpazo}

\usepackage{bbm}
\usepackage{xspace}
\usepackage{algorithmic}
\usepackage{algorithm}
\usepackage{graphicx}
\usepackage{subcaption}

\usepackage{url}            
\usepackage{hyperref}

\definecolor{darkpastelgreen}{rgb}{0.01, 0.75, 0.24}
\definecolor{cadmiumgreen}{rgb}{0.0, 0.42, 0.24}
\definecolor{armygreen}{rgb}{0.29, 0.33, 0.13}
\hypersetup{colorlinks,linkcolor={blue},citecolor={darkpastelgreen},urlcolor={red}}  

\usepackage{amsmath,amsfonts,amsthm,amssymb,color,newclude}
\usepackage{pifont}

\usepackage[english]{babel}
\usepackage[utf8x]{inputenc}
\usepackage[T1]{fontenc} 
\usepackage{scalerel}

\usepackage{amsmath}
\usepackage{amsfonts}
\usepackage{amssymb}
\usepackage{comment}
\usepackage{amsthm}

\newtheorem{assumption}{Assumption}
\numberwithin{assumption}{section}
\newtheorem{theorem}{Theorem}
\numberwithin{theorem}{section}
\newtheorem{corollary}[theorem]{Corollary}
\newtheorem{lemma}[theorem]{Lemma}
\newtheorem{remark}[theorem]{Remark}

\vspace{5mm}
\title{\textbf{Non-asymptotic Superlinear Convergence of Standard Quasi-Newton Methods}
\vspace{3mm}}

\author{Qiujiang Jin\thanks{Department of Electrical and Computer Engineering, The University of Texas at Austin, Austin, TX, USA. \{qiujiang@austin.utexas.edu\}.}\qquad  Aryan Mokhtari\thanks{Department of Electrical and Computer Engineering, The University of Texas at Austin, Austin, TX, USA. \{mokhtari@austin.utexas.edu\}.}}

\date{\empty}

\begin{document}

\maketitle

\vspace{2mm}
\begin{abstract}
\noindent In this paper, we study and prove the non-asymptotic superlinear convergence rate of the Broyden class of quasi-Newton algorithms which includes the Davidon--Fletcher--Powell (DFP) method and the Broyden--Fletcher--Goldfarb--Shanno (BFGS) method. The asymptotic superlinear convergence rate of these quasi-Newton methods has been extensively studied in the literature, but their explicit finite--time local convergence rate is not fully investigated. In this paper, we provide a finite--time (non-asymptotic) convergence analysis for Broyden quasi-Newton algorithms under the assumptions that the objective function is strongly convex, its gradient is Lipschitz continuous, and its Hessian is Lipschitz continuous at the optimal solution. We show that in a local neighborhood of the optimal solution, the iterates generated by both DFP and BFGS converge to the optimal solution at a superlinear rate of $(1/k)^{k/2}$, where $k$ is the number of iterations. We also prove a similar local superlinear convergence result holds for the case that the objective function is self-concordant. Numerical experiments on several datasets confirm our explicit convergence rate bounds. Our theoretical guarantee is one of the first results that provide a non-asymptotic superlinear convergence rate for quasi-Newton methods.

\vspace{5mm}

\end{abstract}

\newpage

\section{Introduction}
\vspace{-1mm}

In this paper, we focus on the non-asymptotic convergence analysis of quasi-Newton methods for the problem of minimizing a convex function $f:\mathbb{R}^d\to \mathbb{R}$, i.e.,
\vspace{-0.5mm}
$$
\min_{x \in \mathbb{R}^d} f(x).
$$
\vspace{-0.5mm}
Specifically, we focus on two different settings.  In the first case, we assume that the objective function~$f$ is strongly convex, smooth (its gradient is Lipschitz continuous), and its Hessian is Lipschitz continuous at the optimal solution. In the second case, we study the setting where the objective function~$f$ is self-concordant. We formally define these settings in the following sections. In both considered cases, the optimal solution solution is unique and denoted by $x_*$. 

There is an extensive literature on the use of first-order methods for convex optimization, and it is well-known that the best achievable convergence rate for first-order methods, when the objective function is  strongly convex and smooth, is a linear convergence rate. Specifically, we say a sequence $\{x_k\}$ converges linearly if  $\|x_k - x_*\|\leq C\gamma^k\|x_0 - x_*\|$, where $\gamma \in (0, 1)$ is the constant of linear convergence, and $C$ is a constant possibly depending on problem parameters. Among first-order methods, the accelerated gradient method proposed in \cite{nesterov1983method} achieves a fast linear convergence rate of   $(1-\sqrt{{\mu}/{L}})^{k/2}$,  where $\mu$ is the strong convexity parameter and $L$ is the smoothness parameter (the Lipschitz constant of the gradient) \cite{nesterov2013introductory}. It is also known that the convergence rate of the accelerated gradient method is optimal for first-order methods in the setting that the problem dimension $d$ is sufficiently larger than the number of iterations \cite{nemirovsky1983problem}.  

Classical alternatives to improve the convergence rate of first-order methods are second-order methods \cite{bennett1916newton,ortega1970iterative,conn2000trust,nesterov2006cubic} and in particular Newton's method. It has been shown that if in addition to smoothness and strong convexity assumptions, the objective function $f$ has  a Lipschitz continuous Hessian, then the iterates generated by Newton's method converge to the optimal solution at a quadratic rate in a local neighborhood of the optimal solution; see \cite[Chapter 9]{boyd04}. A similar result has been established for the case that the objective function is self-concordant \cite{nesterov2004introductory}. Despite the fact that the quadratic convergence rate of Newton's method holds only in a local neighborhood of the optimal solution, it could reduce the overall number of iterations significantly as it is substantially faster than the linear rate of first-order methods. The fast quadratic convergence rate of Newton's method, however, does not come for free. Implementation of Newton's method requires solving a linear system at each iteration with the matrix defined by the objective function Hessian $\nabla^2 f(x)$. As a result, the computational cost of implementing Newton's method in high-dimensional problems is prohibitive, as it could be $\mathcal{O}(d^3)$, unlike first-order methods that have a  per iteration cost of $\mathcal{O}(d)$.

Quasi-Newton algorithms are quite popular since they serve as a middle ground between first-order methods and Newton-type algorithms. They improve the linear convergence rate of first-order methods and achieve a local superlinear rate,  while their computational cost per iteration is $\mathcal{O}(d^2)$ instead of $\mathcal{O}(d^3)$ of Newton's method. The main idea of quasi-Newton methods is to approximate the step of Newton's method without computing the objective function Hessian $\nabla^2 f(x)$ or its inverse $\nabla^2 f(x)^{-1}$ at every iteration \cite[Chapter 6]{nocedal2006numerical}. To be more specific, quasi-Newton methods aim at approximating the curvature of the objective function by using only first-order information of the function, i.e., its gradients $\nabla f(x)$; see Section~\ref{sec:QN_methods} for more details. There are several different approaches for approximating the objective function Hessian and its inverse using first-order information, which leads to different quasi-Newton updates, but perhaps the most popular quasi-Newton algorithms are the Symmetric Rank-One (SR1) method \cite{conn1991convergence}, the Broyden method \cite{broyden1965class,Broyden,gay1979some}, the Davidon-Fletcher-Powell (DFP) method \cite{davidon1959variable,fletcher1963rapidly}, the Broyden-Fletcher-Goldfarb-Shanno (BFGS) method \cite{broyden1970convergence,fletcher1970new,goldfarb1970family,shanno1970conditioning}, and the limited-memory BFGS (L-BFGS) method \cite{nocedal1980updating,liu1989limited}. 

As mentioned earlier, a major advantage of quasi-Newton methods is their asymptotic local superlinear convergence rate. More precisely, we state that the sequence $\{x_k\}$ converges to the optimal solution $x_*$ superlinearly when the ratio between the distance to the optimal solution at time $k+1$ and $k$ approaches zero as $k$ approaches infinity, i.e.,
$$
\lim_{k\to\infty } \frac{\|x_{k+1} - x_*\|}{\|x_k - x_*\|} =0.
$$
For various settings, this superlinear convergence result has been established for a large class of quasi-Newton methods, including the Broyden method \cite{broyden1970convergence,Broyden,more1976global}, the DFP method \cite{powell1971convergence,Broyden,dennis1974characterization}, the BFGS method \cite{Broyden,dennis1974characterization,byrd1987global,gao2019quasi}, and several other variants of these algorithms \cite{griewank1982local,dennis1989convergence,yuan1991modified,al1998global,li1999globally,yabe2007local,mokhtari2017iqn}. Although this result is promising and lies between the linear rate of first-order methods and the quadratic rate of Newton's method, it only holds asymptotically and does not characterize an explicit upper bound on the error of quasi-Newton methods after a finite number of iterations. As a result, the overall complexity of quasi-Newton methods for achieving an $\epsilon$-accurate solution, i.e., $\|x_k - x_*\|\leq \epsilon$, cannot be explicitly characterized. Hence, it is essential to establish a non-asymptotic convergence rate for quasi-Newton methods, which is the main goal of this paper.

In this paper, we show that if the initial iterate is  close to the optimal solution and the initial Hessian approximation error is sufficiently small, then the iterates of the convex Broyden class including both the DFP and BFGS methods converge to the optimal solution at a superlinear rate of $(1/k)^{k/2}$. We further show that our theoretical result suggests a trade-off between the size of the superlinear convergence neighborhood and the rate of superlinear convergence. In other words, one can improve the numerical constant in the above rate at the cost of reducing the radius of the neighborhood in which DFP and BFGS converge superlinearly. We believe that our theoretical guarantee provides one of the first non-asymptotic results for the superlinear convergence rate of BFGS and DFP. 

\noindent \textbf{Related Work.}
In a recent work \cite{rodomanov2020greedy}, the authors studied the non-asymptotic analysis of a class of \textit{greedy} quasi-Newton methods that are based on the update rule of the Broyden family and use a greedily selected basis vectors for updating Hessian approximations. In particular, they show a superlinear convergence rate of $(1-\frac{\mu}{dL})^{k^2/2}(\frac{dL}{\mu})^k$ for this class of algorithms. However, greedy quasi-Newton methods are more computationally costly than standard quasi-Newton methods, as they require computing a greedily selected basis vector at each iteration. It is worth noting that such computation requires access to  additional information beyond the objective function gradient, e.g., the diagonal components of the Hessian.
Also, two recent concurrent papers  study the non-asymptotic superlinear convergence rate of the DFP and BFGS methods \cite{rodomanov2020rates,rodomanov2020ratesnew}. In \cite{rodomanov2020rates}, the authors show that when the objective function is smooth, strongly convex, and strongly self-concordant, the iterates of BFGS and DFP, in a local neighborhood of the optimal solution, achieve the superlinear convergence rate of $(\frac{dL}{\mu k})^{k/2}$ and $(\frac{dL^2}{\mu^2 k})^{k/2}$, respectively. In their follow-up paper \cite{rodomanov2020ratesnew}, they improve the superlinear convergence results to $[e^{\frac{d}{k}\ln{\frac{L}{\mu}}} - 1]^{k/2}$ and $[\frac{L}{\mu}(e^{\frac{d}{k}\ln{\frac{L}{\mu}}} - 1)]^{k/2}$, respectively. 
We would like to highlight that the proof techniques, assumptions, and final theoretical results of \cite{rodomanov2020rates,rodomanov2020ratesnew} and our paper are different and derived independently. The major difference in the analysis is that in \cite{rodomanov2020rates,rodomanov2020ratesnew}, the authors use a potential function related to the trace and the logarithm of the determinant of the Hessian approximation matrix, while we use a Frobenius norm potential function. In addition, our convergence rates for both DFP and BFGS are independent of the problem dimension~$d$. Nevertheless, in our results, the neighborhood of superlinear convergence depends on $d$. Moreover, to derive our results we consider two settings where in the first case the objective function is strongly convex, smooth, and has a Lipschitz continuous Hessain at the optimal solution, and in the second setting the function is self-concordant. Both of these settings are more general than the setting in \cite{rodomanov2020rates,rodomanov2020ratesnew}, which requires the objective function to be strongly convex, smooth, and strongly self-concordant.

\noindent \textbf{Outline.} In Section \ref{sec:QN_methods}, we discuss the Broyden class of quasi-Newton methods, DFP and BFGS. In Section \ref{sec:preliminaries}, we mention our assumptions, notations as well as some general technical lemmas. Then, in Section~\ref{sec:main_results}, we present the main theoretical results of our paper on the non-asymptotic superlinear convergence of DFP and BFGS for the setting that the objective function is strongly convex, smooth, and its Hessian is Lipschitz continuous at the optimal solution. In Section~\ref{sec:self-concordance}, we extend our theoretical results to the class of self-concordant functions, by exploiting the proof techniques developed in Section~\ref{sec:main_results}. In Section~\ref{sec:discussions}, we provide a detailed discussion on the advantages and drawbacks of our theoretical results and compare them with some concurrent works. In Section~\ref{sec:experiments}, we numerically evaluate the performance of DFP and BFGS on several datasets and compare their convergence rates with our theoretical bounds.  Finally, in Section~\ref{sec:conclusion}, we close the paper with some concluding remarks.

\noindent \textbf{Notation.}
For vector $v\in \mathbb{R}^{d}$, its Euclidean norm ($l$-2 norm) is denoted by $\|v\|$.
We denote the Frobenius norm of matrix $A \in \mathbb{R}^{d \times d}$ as $\|A\|_F = \sqrt{\sum_{i = 1}^d\sum_{j = 1}^d A_{ij}^2}$ and its induced $2$-norm is denoted by $\|A\| = \max_{\|v\|=1}\|Av\|$. The trace of matrix $A$, which is the sum of its diagonal elements, is denoted by $\mathrm{Tr}\left(A\right)$.  For any two symmetric matrices $A, B \in \mathbb{R}^{d \times d}$, we denote that $A \preceq B$ if and only if $B - A$ is a symmetric positive-definite matrix. 

\section{Quasi-Newton Methods}\label{sec:QN_methods}

In this section, we review standard quasi-Newton methods, and, in particular, we discuss the updates  of the DFP and BFGS algorithms. Consider a time index $k$, a step size $\eta_k$, and a positive-definite matrix $B_k$ to define a generic descent algorithm through the iteration
\begin{equation}\label{QN_update}
    x_{k+1} = x_k - \eta_k B_k^{-1}\nabla f(x_k).
\end{equation}
Note that if we simply replace $B_k$ by the identity matrix $I$, we recover the update of gradient descent, and if we replace it by the objective function Hessian $\nabla^2 f(x_k)$, we obtain the update of Newton's method. The main goal of quasi-Newton methods is to find a symmetric positive-definite matrix $B_k$ using only first-order information such that $B_k$ is close to the Hessian $\nabla^2 f(x_k)$. Note that the step size $\eta_k$ is often computed according to a line search routine for the global convergence of quasi-Newton methods. Our focus in this paper, however, is on the local convergence of quasi-Newton methods, which requires the unit step size $\eta_k=1$. Hence, in the rest of the paper, we assume that the iterate $x_k$ is sufficiently close to the optimal solution $x_*$ and the step size is $\eta_k = 1$.

In most quasi-Newton methods, the function's curvature is approximated in a way that it satisfies the \textit{secant condition}. To better explain this property, let us first define the variable difference $s_k$ and gradient difference $y_k$ as
\begin{equation}\label{scant_condition}
      s_k = x_{k+1} - x_k,  \quad \text{and} \quad  
      y_k = \nabla f(x_{k+1}) - \nabla f(x_k).
\end{equation}
The goal is to find a matrix $B_{k+1}$ that satisfies the secant condition $ B_{k+1} s_k = y_k$. The rationale for satisfying the secant condition is that the Hessian $\nabla^2 f(x_k)$ approximately satisfies this condition when $x_{k+1}$ and $x_k$ are close to each other, e.g., they are both close to the optimal solution $x_*$. 
However, the secant condition alone is not sufficient to specify $B_{k+1}$. To resolve this indeterminacy, different quasi-Newton algorithms consider different additional  conditions. One common constraint is to enforce the Hessian approximation (or its inverse) at time $k+1$ to be close to the one computed at time $k$. This is a reasonable extra condition as we expect the Hessian (or its inverse) evaluated at $x_{k+1}$ to be close to the one computed at $x_k$. 

In the DFP method, we enforce the proximity condition on Hessian approximations $B_k$ and $B_{k+1}$. Basically, we aim to find the closest positive-definite matrix to $B_k$ (in some weighted matrix norm) that satisfies the secant condition; see Chapter 6 of \cite{nocedal2006numerical} for more details. The update of the Hessian approximation matrices of DFP is given by
\begin{equation}\label{DFP_Hessian_update}
    B^{DFP}_{k+1} = \left(I-\frac{y_k s_k^\top}{y_k^\top s_k}\right) B_k \left(I-\frac{ s_ky_k^\top}{s_k^\top y_k}\right) +\frac{y_k y_k^\top}{y_k^\top s_k}.
\end{equation}
Since implementation of the update in \eqref{QN_update} requires access to the inverse of the Hessian approximation, it is essential to derive an explicit update for the Hessian inverse approximation to avoid the cost of inverting a matrix at each iteration. If we define $H_k$ as the inverse of $B_k$, i.e., $H_k=B_k^{-1}$, using the Sherman-Morrison-Woodbury formula, one can write 
\begin{equation}\label{DFP_Hessian_inverse_update}
  H^{DFP}_{k+1} = H_k - \frac{H_k y_k y_k^\top H_k}{y_k^\top H_k y_k} + \frac{s_k s_k^\top}{s_k^\top y_k}.
\end{equation}

The BFGS method can be considered as the dual of DFP. In BFGS, we also seek a positive-definite matrix that satisfies the secant condition, but instead of forcing the proximity condition on the Hessian approximation $B$, we enforce it on the Hessian inverse approximation $H$. To be more precise, we aim to find a positive-definite matrix $H_{k+1}$ that satisfies the secant condition $  s_k = H_{k+1} y_k$ and is the closest matrix (in some weighted norm) to the previous Hessian inverse approximation $H_k$. The update of the Hessian inverse approximation matrices of BFGS is given by,
\begin{equation}\label{BFGS_Hessian_inverse_update}
    H^{BFGS}_{k+1}= \left(I-\frac{s_k y_k^\top}{y_k^\top s_k}\right) H_k \left(I-\frac{ y_k s_k^\top}{s_k^\top y_k}\right) +\frac{s_k s_k^\top}{y_k^\top s_k}.
\end{equation}
Similarly, by the Sherman-Morrison-Woodbury formula, the update of BFGS method for the Hessian approximation matrices is given by,
\begin{equation}\label{BFGS_Hessian_update}
    B^{BFGS}_{k+1} = B_k - \frac{B_k s_k s_k^\top B_k}{s_k^\top B_k s_k} + \frac{y_k y_k^\top}{s_k^\top y_k}.
\end{equation}
Note that both DFP and BFGS  belong to a more general class of quasi-Newton methods called the Broyden class. The Hessian approximation $B_{k+1}$ of the Broyden class is defined as
\begin{equation}\label{Broyden_Hessian_update}
    B_{k+1} = \phi_k B^{DFP}_{k+1} + (1 - \phi_k) B^{BFGS}_{k+1},
\end{equation}
and the Hessian inverse approximation  is defined as
\begin{equation}\label{Broyden_Hessian_inverse_update}
    H_{k+1} = (1 - \psi_k) H^{DFP}_{k+1} + \psi_k H^{BFGS}_{k+1},
\end{equation}
where $\phi_k, \psi_k \in \mathbb{R}$. In this paper, we only focus on the convex class of Broyden quasi-Newton methods, where $\phi_k, \psi_k \in [0, 1]$. The steps of this class of methods are summarized in Algorithm~\ref{algo_broyden}. In fact, in Algorithm~\ref{algo_broyden}, if we set $\psi_k = 0$, we recover DFP, and if we set $\psi_k = 1$, we recover BFGS. It is worth noting that the cost of  computing the descent direction $H_k \nabla f(x_k)$  for this class of quasi-Newton methods is of $\mathcal{O}(d^2)$, which improves $\mathcal{O}(d^3)$ per iteration cost of Newton's method.

\begin{algorithm}[t]
\caption{The convex Broyden class of quasi-Newton methods}\label{algo_broyden} 
\begin{algorithmic}[1] 
{\REQUIRE Initial iterate  $x_0$ and initial Hessian inverse approximation $H_0$. 
\FOR {$k=0,1,2,\ldots$}
    \STATE Update the variable: $x_{k+1} = x_k - H_k \nabla f(x_k)$;
    \STATE Compute the variable difference $s_k = x_{k+1} - x_k$;
    \IF{$s_k = 0$}
        \STATE Terminate the algorithm
    \ELSE
    \STATE Compute the gradient difference $y_k = \nabla f(x_{k+1}) - \nabla f(x_k) $;
    \STATE Update the Hessian inverse approximation $H_{k+1} = (1 - \psi_k) H^{DFP}_{k+1} + \psi_k H^{BFGS}_{k+1}$ according to \eqref{DFP_Hessian_inverse_update} and \eqref{BFGS_Hessian_inverse_update};
    \ENDIF
\ENDFOR}
\end{algorithmic}\end{algorithm}

\begin{remark}
Note that when $s_k = 0$, we have $\nabla{f(x_k)} = 0$ from \eqref{QN_update} and thus $x_k = x_*$. Hence, in our implementation and analysis we assume $s_k \neq 0$. Moreover, in both considered settings, the objective function is at least strictly convex. As a result, if $s_k \neq 0$, then it follows that $y_k \neq 0$ and $s_k^\top y_k > 0$. This observation shows that the updates of BFGS and DFP are well-defined. Finally, it is well-known that for the convex class of  Broyden methods if $B_{k}$ is symmetric positive-definite and $s_k^\top y_k > 0$, then $B_{k+1}$ is also symmetric positive-definite \cite{nocedal2006numerical}. In Algorithm~\ref{algo_broyden}, we assume that the initial Hessian approximation $B_0$ is symmetric positive-definite, and, hence, all Hessian approximation matrices $B_k$ and their inverse matrices $H_k$ are symmetric positive-definite.
\end{remark}

\section{Preliminaries}\label{sec:preliminaries}

In this section, we first specify the required assumptions for our results in Section~\ref{sec:main_results} and introduce some notations to simplify our expressions. Moreover, we present some intermediate lemmas that will be use later in Section~\ref{sec:main_results} to prove our main theoretical results for the setting that the objective function is strongly convex, smooth, and its Hessian is Lipschitz continuous at the optimal solution. In Section~\ref{sec:self-concordance}, we will use a subset of these intermediate results to extend our analysis to the class of self-concordant functions. 

\subsection{Assumptions}
We formally state the required  assumptions for establishing our theoretical results in Section~\ref{sec:main_results}.
\begin{assumption}\label{ass_str_cvx_smooth}
The objective function $f(x)$ is twice-differentiable. Moreover, it is strongly convex with parameter $\mu > 0$ and its gradient $\nabla f$ is Lipschitz continuous with parameter $L > 0$. Hence, 
\begin{equation}
\mu\| x - y \| \leq \|\nabla{f(x)} - \nabla{f(y)}\| \leq L\| x - y\|, \quad \forall x, y \in \mathbb{R}^{d}.
\end{equation}
\end{assumption}
As $f$ is twice-differentiable, Assumption~\ref{ass_str_cvx_smooth} implies  that the eigenvalues of the Hessian are larger than  $\mu$ and smaller than $L$, i.e., $\mu I \preceq \nabla^{2}{f(x)} \preceq LI, \forall x \in \mathbb{R}^{d}$.

\begin{assumption}\label{ass_Hess_lip}
The Hessian $\nabla^2 f(x)$ satisfies the following condition for some constant $M\geq 0$, 
\begin{equation}
\|\nabla^{2}{f(x)} - \nabla^{2}{f(x_*)}\| \leq M\| x - x_{*}\|, \quad \forall x \in \mathbb{R}^{d}.
\end{equation}
\end{assumption}

The condition in Assumption~\ref{ass_Hess_lip} is common for analyzing second-order methods as we require a regularity condition on the objective function Hessian. In fact,  Assumption~\ref{ass_Hess_lip} is one of the least strict conditions required for the analysis of second-order type methods as it requires Lipschitz continuity of the Hessian only at (near) the optimal solution. This  condition is, indeed, weaker than assuming that the Hessian is Lipschitz continuous everywhere. Note that for the class of strongly convex and smooth functions, the strongly self-concordance assumption required in \cite{rodomanov2020rates,rodomanov2020ratesnew} is equivalent to assuming that the Hessian is Lipschitz continuous everywhere. Hence, the condition in Assumption~\ref{ass_Hess_lip} is also weaker than the one in \cite{rodomanov2020rates,rodomanov2020ratesnew}.  Assumption~\ref{ass_Hess_lip} leads to the following corollary. 

\begin{corollary}\label{corollary}
If the condition in Assumption~\ref{ass_Hess_lip} holds, then for all $ x, y \in \mathbb{R}^{d}$, we have
\begin{equation}\label{corollary_1}
\|\nabla{f(x)} - \nabla{f(y)} - \nabla^{2}{f(x_*)}(x - y)\| \leq \frac{M}{2}\|x - y\|(\|x - x_{*}\| + \|y - x_{*}\|).
\end{equation}
\end{corollary}

\begin{proof}
Check Appendix \ref{sec:proof_corollary}. 
\end{proof}

\begin{remark}
Our analysis can be extended to the case that Assumptions~\ref{ass_str_cvx_smooth} and \ref{ass_Hess_lip} only hold in a local neighborhood of the optimal solution $x_{*}$. Here, we assume they hold in $\mathbb{R}^{d}$ to simplify our proofs.
\end{remark}

\subsection{Notations} 

Next, we briefly mention some of the definitions and notations that will be used in following theorems and proofs. 
We consider $\nabla^{2}f(x_*)^{\frac{1}{2}}$ and $\nabla^{2}f(x_*)^{-\frac{1}{2}}$ as the square root of the matrices $\nabla^{2}f(x_*)$ and $\nabla^{2}f(x_*)^{-1}$, i.e., $\nabla^{2}f(x_*) = \nabla^{2}f(x_*)^{\frac{1}{2}}\nabla^{2}f(x_*)^{\frac{1}{2}}$ and $\nabla^{2}f(x_*)^{-1} = \nabla^{2}f(x_*)^{-\frac{1}{2}}\nabla^{2}f(x_*)^{-\frac{1}{2}}$. By Assumption~\ref{ass_str_cvx_smooth}, both $\nabla^{2}f(x_*)^{\frac{1}{2}}$ and $\nabla^{2}f(x_*)^{-\frac{1}{2}}$ are symmetric positive-definite. Throughout the paper, we analyze and study weighted version of the Hessian approximation $\hat{B}_k$ defined as
\begin{align}\label{main_def_1}
\hat{B}_k = \nabla^{2}f(x_*)^{-\frac{1}{2}}B_k\nabla^{2}f(x_*)^{-\frac{1}{2}}.
\end{align}
$\hat{B}_k$ is symmetric positive-definite, since $B_k$ and $\nabla^{2}f(x_*)^{-\frac{1}{2}}$ are both symmetric positive-definite. We also use $\|\hat{B}_k - I\|_F$ as the measure of closeness between $B_k$ and $\nabla^{2}f(x_*)$, which can be written as
\begin{equation}\label{main_def_2}
   \|\hat{B}_k - I\|_F=  \|\nabla^{2}f(x_*)^{-\frac{1}{2}}\left(B_k - \nabla^{2}f(x_*)\right)\nabla^{2}f(x_*)^{-\frac{1}{2}}\|_F.
\end{equation}
We further define the weighted gradient difference $\hat{y}_k$, the weighted variable difference $\hat{s}_k$, and the weighted gradient $\widehat{\nabla{f}}(x_k)$ as
\begin{align}\label{main_def_3}
\hat{y}_k = \nabla^{2}f(x_*)^{-\frac{1}{2}}y_k, \qquad \hat{s}_k = \nabla^{2}f(x_*)^{\frac{1}{2}}s_k, \qquad \widehat{\nabla{f}}(x_k) = \nabla^{2}f(x_*)^{-\frac{1}{2}}\nabla{f(x_k)}.
\end{align}
To measure closeness to the optimal solution for iterate $x_k$, we use $r_k\in \mathbb{R}^d$, $\sigma_k \in \mathbb{R}$, and $\tau_k \in \mathbb{R}$ which are formally defined as
\begin{align}\label{main_def_4}
    r_k = \nabla^{2}f(x_*)^{\frac{1}{2}}(x_k - x_*), \qquad 
    \sigma_k = \frac{M}{\mu^{\frac{3}{2}}}\|r_k\|, \qquad
    \tau_k = \max\{\sigma_k, \sigma_{k+1}\}.
\end{align}
In \eqref{main_def_4}, $\mu$ is the strong convexity parameter defined in Assumption~\ref{ass_str_cvx_smooth} and $M$ is the Lipschitz continuity parameter of the Hessian at the optimal solution defined in Assumption~\ref{ass_Hess_lip}. In our analysis, we also use the average Hessian $J_k$ and its weighted version $\hat{J}_k$ that are formally defined as
\begin{align}\label{main_def_5}
    J_k = \int_{0}^{1}\nabla^2{f(x_* + \alpha(x_k - x_*))}d\alpha, \qquad 
    \hat{J}_k = \nabla^{2}f(x_*)^{-\frac{1}{2}}J_k\nabla^{2}f(x_*)^{-\frac{1}{2}}.
\end{align}

\subsection{Intermediate Lemmas}
Next, we present some lemmas that we will later use to establish the non-asymptotic superlinear convergence of DFP and BFGS. Proofs of these lemmas are relegated to the appendix. 

\begin{lemma}\label{lemma:lemma_1}
For any matrix $A \in \mathbb{R}^{d \times d}$ and vector $u \in \mathbb{R}^{d}$ with $\|u\| = 1$, we have
\begin{equation}\label{lemma_1_1}
    \|A\|^2_F - \|(I - uu^\top)A(I - uu^\top)\|^2_F \geq \|Au\|^2.
\end{equation}
\end{lemma}

\begin{proof}
Check Appendix \ref{sec:proof_of_lemma_1}. 
\end{proof}

\begin{lemma}\label{lemma:lemma_2}
For any matrices $A, B \in \mathbb{R}^{d \times d}$, we have
\begin{equation}\label{lemma_2_1}
    \|AB\|_F \leq \|A\|\|B\|_F,\qquad 
    \|B^\top AB\|_F \leq \|A\|\|B\|\|B\|_F.
\end{equation}
\end{lemma}

\begin{proof}
Check Appendix \ref{sec:proof_of_lemma_2}. 
\end{proof}

The results in Lemma~\ref{lemma:lemma_1} and Lemma~\ref{lemma:lemma_2} hold for arbitrary matrices. The next lemma focuses on some properties of the weighted average Hessian $\hat{J}_k$ under Assumptions~\ref{ass_str_cvx_smooth} and \ref{ass_Hess_lip}.

\begin{lemma}\label{lemma:lemma_3}
\noindent Recall the definition of $\sigma_k$ in \eqref{main_def_4} and $\hat{J}_k$ in \eqref{main_def_5}. If Assumptions~\ref{ass_str_cvx_smooth} and \ref{ass_Hess_lip} hold, then the following inequalities hold for all $k \geq 0$,
\begin{equation}\label{lemma_3_1}
    \frac{1}{1 + \frac{\sigma_k}{2}}I \preceq \hat{J}_k \preceq (1 + \frac{\sigma_k}{2})I.
\end{equation}
\end{lemma}

\begin{proof}
Check Appendix \ref{sec:proof_of_lemma_3}.
\end{proof}

In the following lemma, we establish some bounds that depend on the weighted gradient difference~$\hat{y}_k$ and the weighted variable difference $\hat{s}_k$.

\begin{lemma}\label{lemma:lemma_4}
Recall the definitions in \eqref{main_def_1} - \eqref{main_def_4}. If Assumptions~\ref{ass_str_cvx_smooth} and \ref{ass_Hess_lip} hold, then the following inequalities hold for all $k \geq 0$,
\begin{equation}\label{lemma_4_1}
    \|\hat{y}_k - \hat{s}_k\| \leq \tau_k\|\hat{s}_k\|,
\end{equation}
\begin{equation}\label{lemma_4_2}
    (1 - \tau_k)\|\hat{s}_k\|^2 \leq \hat{s}_k^\top\hat{y}_k \leq (1 + \tau_k)\|\hat{s}_k\|^2,
\end{equation}
\begin{equation}\label{lemma_4_3}
    (1 - \tau_k)\|\hat{s}_k\| \leq \|\hat{y}_k\| \leq (1 + \tau_k)\|\hat{s}_k\|,
\end{equation}
\begin{equation}\label{lemma_4_4}
    \|\widehat{\nabla{f}}(x_k) - r_k\| \leq \frac{\sigma_k}{2}\|r_k\|.
\end{equation}
\end{lemma}

\begin{proof}
Check Appendix \ref{sec:proof_of_lemma_4}.
\end{proof}

\section{Main Theoretical Results}\label{sec:main_results}

In this section, we characterize the non-asymptotic superlinear convergence of the Broyden class of quasi-Newton methods, when Assumptions~\ref{ass_str_cvx_smooth} and \ref{ass_Hess_lip} hold. In Section~\ref{sec:main_result_hess_appr}, we first establish a crucial proposition which characterizes the error of Hessian approximation for this class of quasi-Newton methods. Then, in Section~\ref{sec:main_result_linear_convergence}, we leverage this result to show that the iterates of this class of algorithms converge at least linearly to the optimal solution, if the initial distance to the optimal solution and the initial Hessian approximation error are sufficiently small. Finally, we use these intermediate results in Section~\ref{sec:main_result_superlinear_convergence} to prove that the iterates of the convex Broyden class, including both DFP and BFGS, converge to the optimal solution at a superlinear rate of $(1/k)^{k/2}$. Note that in Algorithm~\ref{algo_broyden} we use the Hessian inverse approximation matrix $H_k$ to describe the algorithm, but in our analysis we will study the behavior of the Hessian approximation matrix $B_k$.

\subsection{Hessian approximation error: Frobenius norm potential function}\label{sec:main_result_hess_appr}

Next, we use the Frobenius norm of the Hessian approximation error $\|\hat{B}_{k} - I\|_F$ as the potential function in our analysis. Specifically, we will use the results of  Lemma~\ref{lemma:lemma_1}, Lemma~\ref{lemma:lemma_2}, and Lemma~\ref{lemma:lemma_4} to study the dynamic of the  Hessian approximation error $\|\hat{B}_{k} - I\|_F$ for both DFP and BFGS. First, start with the DFP method.

\begin{lemma}\label{lemma:lemma_5}
Consider the update of DFP in \eqref{DFP_Hessian_update} and recall the definition of $\tau_k$ in \eqref{main_def_4}. Suppose that for some $\delta > 0$ and some $k \geq 0$, we have that $\tau_k < 1$ and $\|\hat{B}_k - I\|_{F} \leq \delta$. Then, the matrix $B^{DFP}_{k+1}$ generated by the DFP update  satisfies the following inequality
\begin{equation}\label{claim_Lemma_5}
    \|\hat{B}^{DFP}_{k+1} - I\|_F \leq \|\hat{B}_k - I\|_{F} - \frac{\|(\hat{B}_k - I)\hat{s}_k\|^2}{2\delta\|\hat{s}_k\|^2} + W_k\tau_k,
\end{equation}
where $W_k = \|\hat{B}_k\|\frac{4}{(1 - \tau_k)^2} + \frac{3 + \tau_k}{1 - \tau_k}$.
\end{lemma}

\begin{proof}
The proof and conclusion of this lemma are similar to the ones in Lemma 3.2 in \cite{yabe2007local}, except the value of parameter $W_k$.  This difference comes from the fact that \cite{yabe2007local}  analyzed the modified DFP update, while we consider the standard DFP method. Recall the DFP update in \eqref{DFP_Hessian_update} and multiply both sides of that expression by the matrix $\nabla^{2}f(x_*)^{-\frac{1}{2}}$ from left and right to obtain 
\begin{equation}\label{proof_lemma_5_1}
    \hat{B}^{DFP}_{k+1} = \left(I - \frac{\hat{y}_k\hat{s}_k^\top}{\hat{y}_k^\top\hat{s}_k}\right)\hat{B}_{k}\left(I - \frac{\hat{s}_k\hat{y}_k^\top}{\hat{s}_k^\top\hat{y}_k}\right) +\frac{\hat{y}_k\hat{y}_k^\top}{\hat{y}_k^\top\hat{s}_k},
\end{equation}
where we used the fact that $s_k^\top y_k = s_k^\top \nabla^{2}f(x_*)^{\frac{1}{2}}\nabla^{2}f(x_*)^{-\frac{1}{2}}y_k = \hat{s}_k^\top \hat{y}_k$. To simplify the proof, we use the following notations:
\begin{equation}\label{proof_lemma_5_2}
    B = \hat{B}_{k}, \quad B_+ = \hat{B}^{DFP}_{k+1}, \quad s = \hat{s}_k, \quad y = \hat{y}_k, \quad \tau = \tau_k, \quad P = I - \frac{ss^\top}{\|s\|^2}, \quad Q = \frac{ss^\top}{\|s\|^2} - \frac{sy^\top}{s^\top y}.
\end{equation}
Hence, \eqref{proof_lemma_5_1} is equivalent to
$$
B_+ = \left(I - \frac{ys^\top}{s^\top y}\right)B\left(I - \frac{sy^\top}{s^\top y}\right) + \frac{yy^\top}{s^\top y}.
$$
Moreover, we can express $ B_+ - I$ as
\begin{equation}\label{proof_lemma_5_3}
\begin{split}
    B_+ - I & = (P + Q^\top)B(P + Q) - I + \frac{yy^\top}{s^\top y}\\
    & = PBP + Q^\top BP + PBQ + Q^\top BQ - I + \frac{yy^\top}{s^\top y}\\
    & = P(B - I)P + P^2 - I + \frac{yy^\top}{s^\top y} + Q^\top BP + PBQ + Q^\top BQ.
\end{split}
\end{equation}
Notice that $P^2 = P$ and $P = P^\top$. Thus, \eqref{proof_lemma_5_3} can be simplified as 
$$
    B_+ - I = D + E + G^\top + G + H,
$$
where
$$
    D = P(B - I)P, \quad E = \frac{yy^\top}{s^\top y} - \frac{ss^\top}{\|s\|^2}, \quad G = PBQ, \quad H = Q^\top BQ.
$$
Next, we proceed to upper bound $ \|B_+ - I \|_F$. To do so, we derive upper bounds on the Frobenius norm of matrices $D$, $E$, $G$ and $H$. We start by $\|D\|_F$. If we set $u=s/\|s\|$ and $A=B-I$ in Lemma~\ref{lemma:lemma_1}, we obtain that
\begin{equation}\label{proof_lemma_5_4}
    \frac{\|(B - I)s\|^2}{\|s\|^2} \leq \|B - I\|^2_F - \|D\|^2_F,
\end{equation}
which implies $\|B - I\|^2_F - \|D\|^2_F \geq 0$. Moreover, using the fact that  $a^2 - b^2 \leq 2a(a - b)$  we can write
\begin{equation}\label{proof_lemma_5_5}
    \|B - I\|^2_F - \|D\|^2_F \leq 2\|B - I\|_F(\|B - I\|_F - \|D\|_F) \leq 2\delta(\|B - I\|_F - \|D\|_F),
\end{equation}
where the second inequality follows from the fact that $\|B - I\|^2_F - \|D\|^2_F \geq 0$ and the assumption that  $\|B - I\|_F \leq \delta$. Next, if we replace the right hand side of 
\eqref{proof_lemma_5_4} by its upper bound in  \eqref{proof_lemma_5_5} and massage the resulted expression, we obtain that 
\begin{equation}\label{proof_lemma_5_6}
    \|D\|_F \leq \|B - I\|_F - \frac{\|(B - I)s\|^2}{2\delta\|s\|^2},
\end{equation}
which provides an upper bound on $\|D\|_F$. To derive upper bounds for $\|E\|_F$, $\|G\|_F$ and $\|H\|_F$, we first need to find an upper bound for $\|Q\|_F$, where $Q$ is defined in \eqref{proof_lemma_5_2}. Note that 
$$
\|Q\|_F
=\left\|\frac{ss^\top}{\|s\|^2} - \frac{sy^\top}{s^\top y}\right\|_F
= \left\|\frac{ss^\top}{\|s\|^2} -\frac{sy^\top}{\|s\|^2}+\frac{sy^\top}{\|s\|^2}- \frac{sy^\top}{s^\top y}\right\|_F
\leq \left\|\frac{ss^\top}{\|s\|^2} - \frac{sy^\top}{\|s\|^2}\right\|_F + \left\|\frac{sy^\top}{\|s\|^2} - \frac{sy^\top}{s^\top y}\right\|_F,
$$
where the first equality holds by the definition of $Q$, the second equality is obtained by adding and subtracting $\frac{sy^\top}{\|s\|^2}$, and the inequality holds due to the triangle inequality. We can further simplify the right hand side as
\begin{equation}\label{proof_lemma_5_7}
    \|Q\|_F  \leq \frac{\|s(s - y)^\top\|_F}{\|s\|^2} + \frac{|s^T(s - y)|\|sy^\top\|_F}{\|s\|^2s^\top y} \leq \frac{\|y - s\|}{\|s\|} + \frac{\|y - s\|\|y\|}{s^\top y} \leq \tau + \frac{\tau(1 + \tau)\|s\|^2}{(1 - \tau)\|s\|^2} = \frac{2\tau}{1 - \tau},
\end{equation}
where the second inequality holds using the Cauchy–Schwarz inequality  and the fact that   $\|ab^\top\|_F = \|a\|\|b\|$ for $a, b \in \mathbb{R}^d$, and the last inequality holds due to the results in \eqref{lemma_4_1}, \eqref{lemma_4_2}, and \eqref{lemma_4_3}. 

Next using the upper bound in \eqref{proof_lemma_5_7} on $\|Q\|_F$ we derive an upper bound on $\|E\|_F$. Note that 
$$
\|E\|_F  = \left\|\frac{yy^\top}{s^\top y} - \frac{ss^\top}{\|s\|^2}\right\|_F 
= \left\|\frac{yy^\top}{s^\top y} - \frac{sy^\top}{s^\top y}+\frac{sy^\top}{s^\top y} - \frac{ss^\top}{\|s\|^2}\right\|_F
\leq \left\|\frac{yy^\top}{s^\top y} - \frac{sy^\top}{s^\top y}\right\|_F + \left\|\frac{sy^\top}{s^\top y} - \frac{ss^\top}{\|s\|^2}\right\|_F,
$$
where we used the triangle inequality in the last step. Using the definition of $Q$ we can show that 
\begin{equation}\label{proof_lemma_5_8}
    \|E\|_F 
    \leq \frac{\|(y - s)y^\top\|_F}{s^\top y} + \|Q\|_F \leq \frac{\|y - s\|\|y\|}{s^\top y} + \frac{2\tau}{1 - \tau}\leq \frac{\tau(1 + \tau)\|s\|^2}{(1 - \tau)\|s\|^2} + \frac{2\tau}{1 - \tau} = \frac{3 + \tau}{1 - \tau}\tau,
\end{equation}
where for the second inequality we use \eqref{proof_lemma_5_7} and  $\|ab^\top\|_F = \|a\|\|b\|$, and for the third inequality we use the results in  \eqref{lemma_4_1}, \eqref{lemma_4_2}, and \eqref{lemma_4_3}.

We proceed to derive an upper bound for $\|G\|_F$. Note that $0 \preceq P \preceq I$ and thus $\|P\| \leq 1$. Using this observation, \eqref{proof_lemma_5_7} and the first inequality in \eqref{lemma_2_1}, we can show that $\|G\|_F$ is bounded above by
\begin{equation}\label{proof_lemma_5_9}
    \|G\|_F = \|PBR\|_F \leq \|PB\|\|Q\|_F \leq \|P\|\|B\|\|Q\|_F \leq \|B\|\|Q\|_F \leq \|B\|\frac{2}{1 - \tau}\tau.
\end{equation}
Finally, we provide an upper bound for  $\|H\|_F$. By leveraging the second inequality in \eqref{lemma_2_1} and the fact that $\|A\| \leq \|A\|_F$ for any matrix $A \in \mathbb{R}^{d \times d}$, we can show that
\begin{equation}\label{proof_lemma_5_10}
   \|H\|_F = \|Q^\top BQ\|_F \leq \|B\|\|Q\|\|Q\|_F \leq \|B\|\|Q\|^2_F \leq \|B\|\frac{4\tau^2}{(1 - \tau)^2},
\end{equation}
where for the last inequality we used the result in \eqref{proof_lemma_5_7}.

If we replace $\|D\|_F$, $\|E\|_F$,  $\|G\|_F$, and  $\|H\|_F$ with their upper bounds in \eqref{proof_lemma_5_6}, \eqref{proof_lemma_5_8}, \eqref{proof_lemma_5_9} and \eqref{proof_lemma_5_10}, respectively, we obtain that
\begin{align*}
    \|B_+ - I\|_F & \leq \|D\|_F + \|E\|_F + 2\|G\|_F + \|H\|_F\\
    & \leq \|B - I\|_F - \frac{\|(B - I)s\|^2}{2\delta\|s\|^2} + \frac{3 + \tau}{1 - \tau}\tau + \|B\|\frac{4}{1 - \tau}\tau + \|B\|\frac{4\tau}{(1 - \tau)^2}\tau\\
    & \leq \|B - I\|_F - \frac{\|(B - I)s\|^2}{2\delta\|s\|^2} + W\tau,
\end{align*}
where $W = \|B\|\frac{4}{1 - \tau} + \|B\|\frac{4\tau}{(1 - \tau)^2} + \frac{3 + \tau}{1 - \tau} = \|B\|\frac{4}{(1 - \tau)^2} + \frac{3 + \tau}{1 - \tau}$. Considering the notations introduced in \eqref{proof_lemma_5_2}, the result in \eqref{claim_Lemma_5} follows from the above inequality and the proof is complete.
\end{proof}

The result in Lemma~\ref{lemma:lemma_5} shows how the error of Hessian approximation in DFP evolves as we run the updates. Next, we establish a similar result for the BFGS method.

\begin{lemma}\label{lemma:lemma_6}
Consider the update of BFGS in \eqref{BFGS_Hessian_update} and recall the definition of $\tau_k$ in \eqref{main_def_4}. Suppose that for some $\delta > 0$ and some $k \geq 0$, we have that $\tau_k < 1$ and $\|\hat{B}_k - I\|_{F} \leq \delta$. Then, the matrix $B^{BFGS}_{k+1}$ generated by the BFGS update satisfies the following inequality
\begin{equation}\label{sadadad}
    \|\hat{B}^{BFGS}_{k+1} - I\|_F \leq \|\hat{B}_k - I\|_F - \frac{\hat{s}_k^\top (\hat{B}_k - I) \hat{B}_k (\hat{B}_k - I) \hat{s}_k}{2\delta \hat{s}_k^\top \hat{B}_k \hat{s}_k} + V_k\tau_k,
\end{equation}
where $V_k = \frac{3 + \tau_k}{1 - \tau_k}$.
\end{lemma}

\begin{proof}
The proof of this lemma is adapted from the proof of Lemma~3.6 in \cite{li1999globally}. We should also add that our upper bound in \eqref{sadadad} improves the bound in \cite{li1999globally} as it contains an additional negative term, i.e., $- \frac{\hat{s}_k^\top (\hat{B}_k - I) \hat{B}_k (\hat{B}_k - I) \hat{s}_k}{2\delta \hat{s}_k^\top \hat{B}_k \hat{s}_k}$. Recall the BFGS update in \eqref{BFGS_Hessian_update} and multiply both sides of that expression with  $\nabla^{2}f(x_*)^{-\frac{1}{2}}$ from left and right to obtain 
\begin{equation}\label{proof_lemma_6_1}
    \hat{B}^{BFGS}_{k+1} = \hat{B}_k - \frac{\hat{B}_k \hat{s}_k \hat{s}_k^\top \hat{B}_k}{\hat{s}_k^\top \hat{B}_k \hat{s}_k} + \frac{\hat{y}_k \hat{y}_k^\top}{\hat{s}_k^\top \hat{y}_k},
\end{equation}
where we used the fact that $s_k^\top B_k s_k = s_k^\top \nabla^{2}f(x_*)^{\frac{1}{2}}\nabla^{2}f(x_*)^{-\frac{1}{2}}B_k \nabla^{2}f(x_*)^{-\frac{1}{2}}\nabla^{2}f(x_*)^{\frac{1}{2}} s_k = \hat{s}_k^\top \hat{B}_k \hat{s}_k$. To simplify the proof, we use the following notations:
\begin{equation}\label{proof_lemma_6_2}
    B = \hat{B}_{k}, \quad B_+ = \hat{B}^{BFGS}_{k+1}, \quad s = \hat{s}_k, \quad y = \hat{y}_k, \quad \tau = \tau_k.
\end{equation}
Considering these notations, the expression in  \eqref{proof_lemma_6_1} can be written as
$$
    B_+ = B - \frac{Bss^\top B}{s^\top Bs} + \frac{yy^\top}{s^\top y}.
$$
Moreover, we can show that $ B_+ - I$ is given by
$$
    B_+ - I = B - I + \frac{Bss^\top B}{s^\top Bs} + \frac{ss^\top}{\|s\|^2} + \frac{yy^\top}{s^\top y} - \frac{ss^\top}{\|s\|^2} = D+ E,
$$
where
$$
     D =  B - I + \frac{Bss^\top B}{s^\top Bs} + \frac{ss^\top}{\|s\|^2}, \qquad E = \frac{yy^\top}{s^\top y} - \frac{ss^\top}{\|s\|^2}.
$$
To establish an upper bound on $\|B_+ - I \|_F$, we find upper bounds on $\|D\|^2_F$ and $\|E\|^2_F$. Note that using the fact that $\|D\|^2_F=\mathrm{Tr}\left[DD^\top\right]$ and properties of the trace operator we can show that
\begin{equation}\label{proof_lemma_6_3}
\begin{split}
    \|D\|^2_F & = \mathrm{Tr}\left[\left(B - I - \frac{Bss^\top B}{s^\top Bs} + \frac{ss^\top}{\|s\|^2}\right)\left(B - I - \frac{Bss^\top B}{s^\top Bs} + \frac{ss^\top}{\|s\|^2}\right)^\top\right] \\
    & = \mathrm{Tr}\left[(B - I)^2 - \frac{Bss^\top B(B - I) + (B - I)Bss^\top B}{s^\top Bs} - \frac{ss^\top Bss^\top B + Bss^\top Bss^\top}{s^\top Bs\|s\|^2}\right]\\
    & \quad + \mathrm{Tr}\left[\frac{ss^\top(B - I) + (B - I)ss^\top}{\|s\|^2} + \frac{Bss^\top BBss^\top B}{(s^\top Bs)^2} + \frac{ss^\top ss^\top}{\|s\|^4} \right].\\
    & = \mathrm{Tr}\left[(B - I)^2 - \frac{Bss^\top B(B - I) + (B - I)Bss^\top B}{s^\top Bs} - \frac{ss^\top B + Bss^\top}{\|s\|^2}\right]\\
    & \quad + \mathrm{Tr}\left[\frac{ss^\top(B - I) + (B - I)ss^\top}{\|s\|^2} + \frac{\|Bs\|^2Bss^\top B}{(s^\top Bs)^2} + \frac{ss^\top}{\|s\|^2} \right].\\
\end{split}
\end{equation}
Using the fact that $\mathrm{Tr}\left(ab^\top\right) = a^\top b$ for any $a, b \in \mathbb{R}^d$ we can write the following simplifications:
$$
\mathrm{Tr}\left[\frac{Bss^\top B(B - I) + (B - I)Bss^\top B}{s^\top Bs}\right] = 2\frac{s^\top B(B - I)Bs}{s^\top Bs},
$$
$$
\mathrm{Tr}\left[(B - I)^2\right] = \|B - I\|^2_F,\qquad \mathrm{Tr}\left[\frac{Bss^\top + ss^\top B}{\|s\|^2}\right] = 2\frac{s^\top Bs}{\|s\|^2}, \qquad \mathrm{Tr}\left[\frac{ss^\top}{\|s\|^2}\right] = 1,
$$
$$
\mathrm{Tr}\left[\frac{ss^\top(B - I) + (B - I)ss^\top}{\|s\|^2}\right] = 2\frac{s^\top (B - I)s}{\|s\|^2},\qquad 
\mathrm{Tr}\left[\frac{\|Bs\|^2Bss^\top B}{(s^\top Bs)^2}\right] = \frac{\|Bs\|^4}{(s^\top Bs)^2}.
$$
Substituting the above simplifications into \eqref{proof_lemma_6_3}, we obtain that
\begin{equation}\label{proof_lemma_6_4}
\begin{split}
    \|D\|^2_F & = \|B - I\|^2_F - 2\frac{s^\top B(B - I)Bs}{s^\top Bs} - 2\frac{s^\top Bs}{\|s\|^2} + 2\frac{s^\top(B - I)s}{\|s\|^2} + \frac{\|Bs\|^4}{(s^\top Bs)^2} + 1\\
    & = \|B - I\|^2_F + \left[\left(\frac{\|Bs\|^2}{s^\top Bs}\right)^2 - \frac{s^\top B^3s}{s^\top Bs}\right] - \frac{s^\top(B - I)B(B - I)s}{s^\top Bs}.\\
\end{split}
\end{equation}
Next, we proceed to show that the second term on the right hand side of \eqref{proof_lemma_6_4}, i.e., $\left(\frac{\|Bs\|^2}{s^\top Bs}\right)^2 - \frac{s^\top B^3s}{s^\top Bs}$, is non-positive. Note that by using the Cauchy--Schwarz inequality, we have
$$
\|Bs\|^2 = s^\top B^2s = s^\top B^{\frac{3}{2}}B^{\frac{1}{2}}s \leq \|B^{\frac{3}{2}}s\|\|B^{\frac{1}{2}}s\|.
$$
Now by computing the squared of both sides we obtain
$
  \|Bs\|^4 \leq \|B^{\frac{3}{2}}s\|^2\|B^{\frac{1}{2}}s\|^2 = s^\top B^3 ss^\top Bs,
$
which implies that 
\begin{equation}\label{proof_lemma_6_5}
 \left(\frac{\|Bs\|^2}{s^\top Bs}\right)^2 - \frac{s^\top B^3s}{s^\top Bs} \leq 0.
\end{equation}
By combining \eqref{proof_lemma_6_4} and \eqref{proof_lemma_6_5}, we obtain that 
\begin{equation}\label{proof_lemma_6_6}
    \frac{s^\top(B - I)B(B - I)s}{s^\top Bs} \leq \|B- I\|^2_F - \|D\|^2_F.
\end{equation}
The above inequality implies that $\|B - I\|^2_F - \|D\|^2_F \geq 0$. Moreover, using the fact that $a^2 - b^2 \leq 2a(a - b), \forall a,b \in \mathbb{R}$, we can show that
\begin{equation}\label{proof_lemma_6_7}
    \|B - I\|^2_F - \|D\|^2_F \leq 2\|B - I\|_F(\|B - I\|_F - \|D_k\|_F) \leq 2\delta(\|B - I\|_F - \|D\|_F).
\end{equation}
where the second inequality follows from $\|B - I\|^2_F - \|D\|^2_F \geq 0$ and the fact that 
$\|B - I\|_F \leq \delta$. Now if combine the results in \eqref{proof_lemma_6_6} and \eqref{proof_lemma_6_7}, we obtain that
\begin{equation}\label{proof_lemma_6_8}
    \|D\|_F \leq \|B - I\|_F - \frac{s^\top(B - I)B(B - I)s}{2\delta s^\top Bs},
\end{equation}
which provides an upper bound on $\|D\|_F$. Moreover, according to  \eqref{proof_lemma_5_8}, $\|E\|_F$ is bounded above by  
\begin{equation}\label{proof_lemma_6_9}
    \|E\|_F \leq \frac{3 + \tau}{1 - \tau}\tau.
\end{equation}
If we replace $\|D\|_F$ and $\|E\|_F$ with their upper bounds in \eqref{proof_lemma_6_8} and \eqref{proof_lemma_6_9}, we obtain that
$$
    \|B_+ - I\|_F \leq \|D\|_F + \|E\|_F \leq \|B - I\|_F - \frac{s^\top(B - I)B(B - I)s}{2\delta s^\top Bs} + V\tau,
$$
where $V = \frac{3 + \tau}{1 - \tau}$. Considering the notations in \eqref{proof_lemma_6_2}, the claim follows from the above inequality.
\end{proof}

Now we can combine Lemma \ref{lemma:lemma_5} and Lemma \ref{lemma:lemma_6} to derive a bound on the error of Hessian approximation for the (convex) Broyden class of quasi-Newton methods.

\begin{lemma}\label{lemma:lemma_7}
Consider the update of the (convex) Broyden family in \eqref{Broyden_Hessian_update} and recall the definition of $\tau_k$ in \eqref{main_def_4}. Suppose that for some $\delta > 0$ and some $k \geq 0$, we have that $\tau_k < 1$ and $\|\hat{B}_k - I\|_{F} \leq \delta$. Then, the matrix $B_{k+1}$ generated by  \eqref{Broyden_Hessian_update} satisfies the following inequality
\begin{equation}\label{lemma_7_1}
    \|\hat{B}_{k+1} - I\|_F \leq \|\hat{B}_k - I\|_F - \phi_k\frac{\|(\hat{B}_k - I)\hat{s}_k\|^2}{2\delta\|\hat{s}_k\|^2} - (1 - \phi_k)\frac{\hat{s}_k^\top (\hat{B}_k - I) \hat{B}_k (\hat{B}_k - I) \hat{s}_k}{2\delta \hat{s}_k^\top \hat{B}_k \hat{s}_k} + Z_k\tau_k,
\end{equation}
where $Z_k = \phi_k\|\hat{B}_k\|\frac{4}{(1 - \tau_k)^2} + \frac{3 + \tau_k}{1 - \tau_k}$. We also have that
\begin{equation}\label{lemma_7_2}
    \|\hat{B}_{k+1} - I\|_F \leq \|\hat{B}_k - I\|_F + Z_k\tau_k.
\end{equation}
\end{lemma}
\begin{proof}
Notice that $B_{k+1} = \phi_k B^{DFP}_{k+1} + (1 - \phi_k) B^{BFGS}_{k+1}$. Using this expression and the convexity of the norm, we can show that 
$$
 \|\hat{B}_{k+1} - I\|_F=
 \|\phi_k B^{DFP}_{k+1} + (1 - \phi_k) B^{BFGS}_{k+1}-I\|_F\leq \phi_k \|B^{DFP}_{k+1}-I\|_F+(1 - \phi_k)\| B^{BFGS}_{k+1}-I\|_F.
$$
By replacing $\|B^{DFP}_{k+1}-I\|_F$ and $\| B^{BFGS}_{k+1}-I\|_F$ with their upper bounds in  Lemma \ref{lemma:lemma_5} and Lemma \ref{lemma:lemma_6},  the claim in \eqref{lemma_7_1} follows. Moreover, since $\phi_k \in [0,1]$, $\delta > 0$, $\frac{\|(\hat{B}_k - I)\hat{s}_k\|^2}{\|\hat{s}_k\|^2} \geq 0$ and $\frac{\hat{s}_k^\top (\hat{B}_k - I) \hat{B}_k (\hat{B}_k - I) \hat{s}_k}{\hat{s}_k^\top \hat{B}_k \hat{s}_k} \geq 0$, the result in \eqref{lemma_7_1} implies \eqref{lemma_7_2}.
\end{proof}

\subsection{Linear convergence}\label{sec:main_result_linear_convergence}

In this section, we leverage the results from the previous section on the error of Hessian approximation to show that if the initial iterate is sufficiently close to the optimal solution and the initial Hessian approximation matrix is close  to the Hessian at the optimal solution, the iterates of BFGS and DFP converge at least linearly to the optimal solution. Moreover, the Hessian approximation matrices always stay close to the Hessian at the optimal solution and the norms of Hessian approximation matrix and its inverse are always bounded above. These results are essential in proving our non-asymptotic superlinear convergence results.

\begin{lemma}\label{lemma:lemma_8}
Consider the convex Broyden class of quasi-Newton methods described in Algorithm~\ref{algo_broyden}, and recall the definitions in \eqref{main_def_1}-\eqref{main_def_4}. Suppose Assumptions~\ref{ass_str_cvx_smooth} and \ref{ass_Hess_lip} hold. Moreover, suppose  the initial point $x_0$ and initial Hessian approximation matrix $B_0$ satisfy
\begin{equation}\label{lemma_8_1}
    \sigma_0 \leq \epsilon, \qquad \|\hat{B}_0 - I\|_F \leq \delta,
\end{equation}
where $\epsilon, \delta \in (0, \frac{1}{2})$ such that for some $\rho \in (0, 1)$, they satisfy
\begin{equation}\label{lemma_8_2}
\max_{k \geq 0}{\left[\phi_k\frac{4(2\delta + 1)}{(1 - \epsilon)^2} + \frac{3 + \epsilon}{1 - \epsilon}\right]}\frac{\epsilon}{1 - \rho} \leq \delta, \qquad \frac{\epsilon}{2} + 2\delta \leq (1 - 2\delta)\rho \ .
\end{equation}
Then, the sequence of iterates $\{x_k\}_{k=0}^{+\infty}$ converges to the optimal solution $x_*$ with
\begin{equation}\label{lemma_8_3}
\sigma_{k + 1} \leq \rho\sigma_{k}, \qquad \forall k \geq 0.
\end{equation}
Furthermore, the matrices $\{B_k\}_{k=0}^{+\infty}$ stay in a neighborhood of  $\nabla^{2}{f(x_*)}$ defined as
\begin{equation}\label{lemma_8_4}
\|\hat{B}_{k} - I\|_F \leq 2\delta, \qquad \forall k \geq 0.
\end{equation}
Moreover, the norms $\{\|\hat{B}_k\|\}_{k=0}^{+\infty}$ and $\{\|\hat{B}_k^{-1}\|\}_{k=0}^{+\infty}$ are all uniformly bounded above by
\begin{equation}\label{lemma_8_5}
\|\hat{B}_k\| \leq 1 + 2\delta, \qquad \|\hat{B}_k^{-1}\| \leq \frac{1}{1 - 2\delta}, \qquad \forall k \geq 0.
\end{equation}
\end{lemma}

\begin{proof}
The proof of this lemma is adapted from the proof of Theorem 3.1 in \cite{yabe2007local}. In \cite{yabe2007local}, the authors prove the results for the modified DFP method, while we consider the more general class of Broyden methods. We will use induction to prove \eqref{lemma_8_3}, \eqref{lemma_8_4} and \eqref{lemma_8_5}. First consider the base case of $k = 0$. By the initial condition \eqref{lemma_8_1}, it's obvious that \eqref{lemma_8_4} holds for $k = 0$. From \eqref{lemma_8_4} we know that all the eigenvalues of $\hat{B}_0$ are in the interval $[1 - 2\delta, 1 + 2\delta]$. Suppose that $\lambda_{max}(\hat{B}_0)$ is the largest eigenvalue of $\hat{B}_0$ and $\lambda_{min}(\hat{B}_0)$ is the smallest eigenvalue of $\hat{B}_0$, we have
$$
\|\hat{B}_0\| = \lambda_{max}(\hat{B}_0) \leq 1 + 2\delta, \quad \|\hat{B}_0^{-1}\| = \frac{1}{\lambda_{min}(\hat{B}_0)} \leq \frac{1}{1 - 2\delta}.
$$
Hence, \eqref{lemma_8_5} holds for $k = 0$. Based on Assumptions~\ref{ass_str_cvx_smooth}-\ref{ass_Hess_lip} and the definitions in \eqref{main_def_1}-\eqref{main_def_4}, we have
\begin{equation}\label{proof_lemma_8_1}
\begin{split}
    \sigma_1 & = \frac{M}{\mu^{\frac{3}{2}}}\|\nabla^{2}{f(x_*)}^{\frac{1}{2}}(x_1 - x_*)\| = \frac{M}{\mu^{\frac{3}{2}}}\|\nabla^{2}{f(x_*)}^{\frac{1}{2}}(x_0 - B_0^{-1}\nabla{f(x_0)} - x_*)\|\\
    & = \frac{M}{\mu^{\frac{3}{2}}}\|\nabla^{2}{f(x_*)}^{\frac{1}{2}}B_0^{-1}[\nabla{f(x_0)} - \nabla^2{f(x_*)}(x_0 - x_*) - (B_0 - \nabla^2{f(x_*)})(x_0 - x_*)]\|\\
    & = \frac{M}{\mu^{\frac{3}{2}}}\|\hat{B}_0^{-1}[\widehat{\nabla{f}}(x_0) - r_0 - (\hat{B}_0 - I)r_0]\| \leq \frac{M}{\mu^{\frac{3}{2}}}\|\hat{B}_0^{-1}\|\left(\|\widehat{\nabla{f}}(x_0) - r_0\| + \|\hat{B}_0 - I\|\|r_0\|\right).
\end{split}
\end{equation}
Now using the result in \eqref{lemma_4_4}, and the bounds in \eqref{lemma_8_1}, \eqref{lemma_8_2}, \eqref{lemma_8_4} and \eqref{lemma_8_5} for $k = 0$, we can write
$$
\sigma_1 \leq \frac{M}{\mu^{\frac{3}{2}}}\|\hat{B}_0^{-1}\|(\frac{\sigma_0}{2}\|r_0\| + \|\hat{B}_0 - I\|\|r_0\|) = \|\hat{B}_0^{-1}\|(\frac{\sigma_0}{2} + \|\hat{B}_0 - I\|)\sigma_0 \leq  \frac{1}{1 - 2\delta}(\frac{\epsilon}{2} + 2\delta)\sigma_0 \leq \rho\sigma_0.
$$
This indicates that the condition in \eqref{lemma_8_3} holds for $k = 0$. Hence, all the conditions in \eqref{lemma_8_3}, \eqref{lemma_8_4} and \eqref{lemma_8_5} hold for $k = 0$, and the base of induction is complete. Now we assume that the conditions in \eqref{lemma_8_3}, \eqref{lemma_8_4} and \eqref{lemma_8_5} hold for all $0 \leq k \leq t$, where $t \geq 0$. Our goal is to show that these conditions are also satisfied for the case of $k = t + 1$. Since \eqref{lemma_8_3} holds for all $0\leq k \leq t$, we have $\tau_k = \max\{\sigma_k, \sigma_{k + 1}\} = \sigma_k \leq \epsilon < 1$ for $0 \leq k \leq t$. Moreover, since the condition in \eqref{lemma_8_4} holds for $0\leq k\leq t$, we know that $\|\hat{B}_k-I\|_F \leq 2\delta$ for $0 \leq k \leq t$. Hence, by \eqref{lemma_7_2} in Lemma \ref{lemma:lemma_7}, we obtain that
\begin{equation}\label{proof_lemma_8_2}
    \|\hat{B}_{k+1} - I\|_F \leq \|\hat{B}_k - I\|_{F} + Z_k\sigma_k, \qquad 0 \leq k \leq t,
\end{equation}
where $Z_k = \phi_k\|\hat{B}_k\|\frac{4}{(1 - \sigma_k)^2} + \frac{3 + \sigma_k}{1 - \sigma_k}$. Using \eqref{lemma_8_5} and $\sigma_k \leq \epsilon$ for $0 \leq k \leq t$, we obtain that
$$
Z_k \leq\phi_k \frac{4(2\delta + 1)}{(1 - \epsilon)^2} + \frac{3 + \epsilon}{1 - \epsilon}, \qquad 0 \leq k \leq t.
$$
Further if \eqref{lemma_8_1} and \eqref{lemma_8_3} hold for $0 \leq k \leq t$, we have that
\begin{equation}\label{proof_lemma_8_3}
    \sum_{k = 0}^{t}\sigma_k \leq \sum_{k = 0}^{t}\rho^{k}\sigma_0 \leq \frac{\sigma_0}{1 - \rho} \leq \frac{\epsilon}{1 - \rho}.
\end{equation}
Considering these results we can show that
\begin{equation}\label{proof_lemma_8_4}
\begin{split}
    \sum_{k = 0}^{t}Z_k\sigma_k 
    & \leq \max_{0 \leq k \leq t}{\left[\phi_k\frac{4(2\delta + 1)}{(1 - \epsilon)^2} + \frac{3 + \epsilon}{1 - \epsilon}\right]}\sum_{k = 0}^{t}\sigma_k \leq \max_{0 \leq k \leq t}{\left[\phi_k\frac{4(2\delta + 1)}{(1 - \epsilon)^2} + \frac{3 + \epsilon}{1 - \epsilon}\right]}\frac{\epsilon}{1 - \rho} \leq \delta,
\end{split}
\end{equation}
where the last inequality holds due to the first inequality in \eqref{lemma_8_2}. By leveraging \eqref{proof_lemma_8_4} and \eqref{lemma_8_1} and computing the sum of the terms in the left and right hand side of \eqref{proof_lemma_8_2} from $k = 0$ to $t$, we obtain
$$
\|\hat{B}_{t + 1} - I\|_F \leq \|\hat{B}_0 - I\|_{F} + \sum_{k = 0}^{t}Z_k\sigma_k \leq \delta + \delta = 2\delta,
$$
which implies that \eqref{lemma_8_4} holds for $k = t + 1$. Applying the same techniques we used in the base case, we can prove that \eqref{lemma_8_3} and \eqref{lemma_8_5} hold for $k = t + 1$. Hence, all the claims in \eqref{lemma_8_3}, \eqref{lemma_8_4} and \eqref{lemma_8_5} hold for $k = t + 1$, and our induction step is complete.
\end{proof}

\subsection{Explicit non-asymptotic superlinear rate}\label{sec:main_result_superlinear_convergence}

In the previous section, we established local linear convergence of iterates generated by the convex Broyden class including DFP and BFGS. Indeed, these local linear results are not our ultimate goal, as first-order methods are also linearly convergent under the same assumptions. However, the linear convergence is required to establish a local non-asymptotic superlinear convergence result, which is our main contribution. 
Next, we state the main results of this paper on the non-asymptotic superlinear convergence rate of the convex Broyden class of quasi-Newton methods. To prove this claim, we use the results in Lemma~\ref{lemma:lemma_7} and Lemma~\ref{lemma:lemma_8}.

\begin{theorem}\label{thm_broyden}
Consider the convex Broyden class of quasi-Newton methods described in Algorithm~\ref{algo_broyden}. Suppose the objective function $f$ satisfies the conditions in Assumptions~\ref{ass_str_cvx_smooth} and \ref{ass_Hess_lip}. Moreover, suppose  the initial point $x_0$ and initial Hessian approximation matrix $B_0$ satisfy
\begin{equation}\label{thm_broyden_1}
    \frac{M}{\mu^{\frac{3}{2}}}\|\nabla^{2}f(x_*)^\frac{1}{2}(x_0 - x_*)\| \leq \epsilon, \qquad \|\nabla^{2}f(x_*)^{-\frac{1}{2}}\ \! (B_0 - \nabla^{2}f(x_*))\ \!\nabla^{2}f(x_*)^{-\frac{1}{2}}\|_F \leq \delta,
\end{equation}
where $\epsilon, \delta \in (0, \frac{1}{2})$ such that for some $\rho \in (0, 1)$, they satisfy
\begin{equation}\label{thm_broyden_2}
\max_{k \geq 0}{\left[\phi_k(2\delta + 1)\frac{4}{(1 - \epsilon)^2} + \frac{3 + \epsilon}{1 - \epsilon}\right]}\frac{\epsilon}{1 - \rho} \leq \delta, \qquad \frac{\epsilon}{2} + 2\delta \leq (1 - 2\delta)\rho \ .
\end{equation}
Then the iterates $\{x_{k}\}_{k=0}^{+\infty}$ generated by the convex Broyden class of quasi-Newton methods converge to $x_*$ at a superlinear rate of
\begin{equation}\label{thm_broyden_3}
\frac{\|\nabla^{2}f(x_*)^\frac{1}{2}(x_k - x_*)\|}{\|\nabla^{2}f(x_*)^\frac{1}{2}(x_0 - x_*)\|} \leq \left(\frac{2\sqrt{2}\delta(1 + \rho)(1 + \frac{\epsilon}{2})q\sqrt{k} + \frac{(1 + \rho)(1 + \frac{\epsilon}{2})\epsilon}{2(1 - \rho)}}{k}\right)^k, \qquad \forall k \geq 1,
\end{equation}
\begin{equation}\label{thm_broyden_4}
\frac{f(x_k) - f(x_*)}{f(x_0) - f(x_*)} \leq \left(1 + \frac{\epsilon}{2}\right)^2\left(\frac{2\sqrt{2}\delta(1 + \rho)(1 + \frac{\epsilon}{2})q\sqrt{k} + \frac{(1 + \rho)(1 + \frac{\epsilon}{2})\epsilon}{2(1 - \rho)}}{k}\right)^{2k}, \qquad \forall k \geq 1,
\end{equation}
where $q = \max_{k \geq 0}{\sqrt{\frac{1}{\phi_k + (1 - \phi_k)\frac{1 - 2\delta}{1 + 2\delta}}}} \in \left[1, \sqrt{\frac{1 + 2\delta}{1 - 2\delta}}\right]$.
\end{theorem}
\begin{proof}
When both conditions \eqref{thm_broyden_1} and \eqref{thm_broyden_2} hold, by Lemma \ref{lemma:lemma_8}, the results in \eqref{lemma_8_3}, \eqref{lemma_8_4} and \eqref{lemma_8_5} hold. This indicates that for any $t \geq 0$, we have
$$
\tau_t = \max\{\sigma_t, \sigma_{t+1}\} = \sigma_t \leq \sigma_0 \leq \epsilon < 1,\qquad
\|\hat{B}_t - I\|_F \leq 2\delta.
$$
Hence, using Lemma \ref{lemma:lemma_7} for any $t \geq 0$, we can show that
\begin{equation}\label{proof_thm_broyden_1}
   \|\hat{B}_{t+1} - I\|_F \leq \|\hat{B}_t - I\|_F - \phi_t\frac{\|(\hat{B}_t - I)\hat{s}_t\|^2}{4\delta\|\hat{s}_t\|^2} - (1 - \phi_t)\frac{\hat{s}_t^\top (\hat{B}_t - I) \hat{B}_t (\hat{B}_t - I) \hat{s}_t}{4\delta \hat{s}_t^\top \hat{B}_t \hat{s}_t} + Z_t\sigma_t,
\end{equation}
where $Z_t = \phi_t\|\hat{B}_t\|\frac{4}{(1 - \sigma_t)^2} + \frac{3 + \sigma_t}{1 - \sigma_t}$. Using \eqref{proof_lemma_8_3} and \eqref{proof_lemma_8_4}, for $k \geq 0$ we have
\begin{equation}\label{proof_thm_broyden_2}
    \sum_{t = 0}^{k}\sigma_t \leq \frac{\epsilon}{1 - \rho}, \qquad \sum_{t = 0}^{k}Z_t \sigma_t \leq \delta.
\end{equation}
Now compute the sum of both sides of \eqref{proof_thm_broyden_1} from $t = 0$ to $k - 1$ to obtain
$$
\|\hat{B}_k - I\|_F \leq \|\hat{B}_0 - I\|_F - \sum_{t = 0}^{k - 1}\left[\phi_t\frac{\|(\hat{B}_t - I)\hat{s}_t\|^2}{4\delta\|\hat{s}_t\|^2} + (1 - \phi_t)\frac{\hat{s}_t^\top (\hat{B}_t - I) \hat{B}_t (\hat{B}_t - I) \hat{s}_t}{4\delta \hat{s}_t^\top \hat{B}_t \hat{s}_t}\right] + \sum_{t = 0}^{k - 1}Z_t\sigma_t.
$$
Regroup the terms and use the results in \eqref{thm_broyden_1} and \eqref{proof_thm_broyden_2} to show that
\begin{align*}
& \sum_{t = 0}^{k - 1}\left[\phi_t\frac{\|(\hat{B}_t - I)\hat{s}_t\|^2}{4\delta\|\hat{s}_t\|^2} + (1 - \phi_t)\frac{\hat{s}_t^\top (\hat{B}_t - I) \hat{B}_t (\hat{B}_t - I) \hat{s}_t}{4\delta \hat{s}_t^\top \hat{B}_t \hat{s}_t}\right]\\ 
\leq & \|\hat{B}_0 - I\|_F - \|\hat{B}_k - I\|_F + \sum_{t = 0}^{k - 1}Z_t\sigma_t \nonumber \leq \|\hat{B}_0 - I\|_F + \sum_{t = 0}^{k - 1}Z_t\sigma_t\nonumber \leq \delta + \delta = 2\delta,
\end{align*}
which leads to
\begin{equation}\label{proof_thm_broyden_3}
\sum_{t = 0}^{k - 1}\left[\phi_t\frac{\|(\hat{B}_t - I)\hat{s}_t\|^2}{\|\hat{s}_t\|^2} + (1 - \phi_t)\frac{\hat{s}_t^\top (\hat{B}_t - I) \hat{B}_t (\hat{B}_t - I) \hat{s}_t}{\hat{s}_t^\top \hat{B}_t \hat{s}_t}\right] \leq 8\delta^2.
\end{equation}
Moreover, using the bounds in \eqref{lemma_8_5} we can show that
$$
\hat{s}_t^\top (\hat{B}_t - I) \hat{B}_t (\hat{B}_t - I) \hat{s}_t \geq \frac{1}{\|\hat{B}_t^{-1}\|}\|(\hat{B}_t - I) \hat{s}_t\|^2 \geq (1 - 2\delta)\|(\hat{B}_t - I) \hat{s}_t\|^2,
$$
$$
\hat{s}_t^\top \hat{B}_t \hat{s}_t \leq \|\hat{B}_t\|\|\hat{s}_t\|^2 \leq (1 + 2\delta)\|\hat{s}_t\|^2.
$$
Hence, we have 
\begin{equation}\label{proof_thm_broyden_4}
\frac{\hat{s}_t^\top (\hat{B}_t - I) \hat{B}_t (\hat{B}_t - I) \hat{s}_t}{\hat{s}_t^\top \hat{B}_t \hat{s}_t} \geq \frac{1 - 2\delta}{1 + 2\delta}\frac{\|(\hat{B}_t - I)\hat{s}_t\|^2}{\|\hat{s}_t\|^2}.
\end{equation}
By combining the bounds in \eqref{proof_thm_broyden_3} and \eqref{proof_thm_broyden_4}, we obtain
\begin{align*}
\sum_{t = 0}^{k - 1}\left[\phi_t + (1 - \phi_t)\frac{1 - 2\delta}{1 + 2\delta}\right] \frac{\|(\hat{B}_t - I)\hat{s}_t\|^2}{\|\hat{s}_t\|^2}
\leq 8\delta^2.
\end{align*}
Now by computing the minimum value of the term $\phi_t + (1 - \phi_t)\frac{1 - 2\delta}{1 + 2\delta}$, we can show
\begin{align*}
     \min_{k \geq 0}{\left[\phi_k + (1 - \phi_k)\frac{1 - 2\delta}{1 + 2\delta}\right]}\sum_{t = 0}^{k - 1}\frac{\|(\hat{B}_t - I)\hat{s}_t\|^2}{\|\hat{s}_t\|^2} \leq 8\delta^2,
\end{align*}
and by regrouping the terms, we obtain that
$$
\sum_{t = 0}^{k - 1}\frac{\|(\hat{B}_t - I)\hat{s}_t\|^2}{\|\hat{s}_t\|^2} \leq 8\delta^2\frac{1}{\min_{k \geq 0}{\left[\phi_k + (1 - \phi_k)\frac{1 - 2\delta}{1 + 2\delta}\right]}} = 8\delta^2\max_{k \geq 0}{\frac{1}{\phi_k + (1 - \phi_k)\frac{1 - 2\delta}{1 + 2\delta}}}.
$$
Considering the definition 
$
q := \max_{k \geq 0}{\sqrt{\frac{1}{\phi_k + (1 - \phi_k)\frac{1 - 2\delta}{1 + 2\delta}}}},
$
we can simplify our upper bound as
$$
\sum_{t = 0}^{k - 1}\frac{\|(\hat{B}_t - I)\hat{s}_t\|^2}{\|\hat{s}_t\|^2} \leq 8\delta^2q^2.
$$
By using the Cauchy-Schwarz inequality, we obtain that
\begin{equation}\label{proof_thm_broyden_5}
    \sum_{t = 0}^{k - 1}\frac{\|(\hat{B}_t - I)\hat{s}_t\|}{\|\hat{s}_t\|} \leq 2\sqrt{2}\delta q\sqrt{k}.
\end{equation}
Note that since $\phi_k \in [0, 1]$, we have $q \in \left[1, \sqrt{\frac{1 + 2\delta}{1 - 2\delta}}\right]$. The result in \eqref{proof_thm_broyden_5} provides an upper bound on $\sum_{t = 0}^{k - 1}\frac{\|(\hat{B}_t - I)\hat{s}_t\|}{\|\hat{s}_t\|} $, which is a crucial term in the remaining of our proof.

Now, note that $\nabla{f(x_t)} = J_t(x_t - x_*)$, where $J_t$ is defined in \eqref{main_def_5}. This implies that $ x_t - x_*=J_t^{-1}\nabla{f(x_t)}$ and hence  we have
$$
x_{t + 1} - x_* = x_t - x_* + s_t 
=J_t^{-1}\nabla{f(x_t)}+ s_t 
= - J_t^{-1}B_t s_t + s_t = J_t^{-1}(J_t - B_t)s_t.
$$
where the third equality holds since $-B_t s_t = \nabla{f(x_t)} $. Pre-multiply both sides of the above expression by $\nabla^2{f(x_*)}^{\frac{1}{2}}$  to obtain
$$
r_{t + 1} = \hat{J}_t^{-1}(\hat{J}_t - \hat{B}_t)\hat{s}_t = \hat{J}_t^{-1}[(\hat{J}_t - I)\hat{s}_t - (\hat{B}_t - I)\hat{s}_t].
$$
Therefore, we obtain that
$$
\|r_{t + 1}\| \leq \|\hat{J}_t^{-1}\|\left(\|(\hat{J}_t - I)\hat{s}_t\| + \|(\hat{B}_t - I)\hat{s}_t\|\right) \leq \|\hat{J}_t^{-1}\|\left(\|\hat{J}_t - I\| + \frac{\|(\hat{B}_t - I)\hat{s}_t\|}{\|\hat{s}_t\|}\right)\|\hat{s}_t\|.
$$
From Lemma~\ref{lemma:lemma_3} we know that $\|\hat{J}_t^{-1}\| \leq 1 + \frac{\sigma_t}{2}$ and $\|\hat{J}_t - I\| \leq \frac{\sigma_t}{2}$. Therefore, we have
\begin{equation}\label{proof_thm_broyden_6}
    \|r_{t + 1}\| \leq \left(1 + \frac{\sigma_t}{2}\right)\left(\frac{\sigma_t}{2} + \frac{\|(\hat{B}_t - I)\hat{s}_t\|}{\|\hat{s}_t\|}\right)\|\hat{s}_t\|.
\end{equation}
Also, since $\sigma_{t+1} \leq \rho\sigma_t$ and $\sigma_t = \frac{M}{\mu^{\frac{3}{2}}}\|r_t\|$, we obtain that $\|r_{t+1}\| \leq \rho\|r_t\|$. Hence, we can write 
\begin{equation}\label{proof_thm_broyden_7}
    \|\hat{s}_t\| = \|\nabla{f(x_*)}^{\frac{1}{2}}(x_{t+1} - x_* + x_* - x_t)\| \leq \|r_{t+1}\| + \|r_t\| \leq (1 + \rho)\|r_t\|.
\end{equation}
Using the expressions in \eqref{proof_thm_broyden_6} and \eqref{proof_thm_broyden_7}, we can show that $\frac{\|r_{t+1}\|}{\|r_t\|}$ is bounded above by 
\begin{equation}\label{proof_thm_broyden_8}
\frac{\|r_{t+1}\|}{\|r_t\|} \leq (1 + \rho)\left(1 + \frac{\sigma_t}{2}\right)\left(\frac{\sigma_t}{2} + \frac{\|(\hat{B}_t - I)\hat{s}_t\|}{\|\hat{s}_t\|}\right).
\end{equation}
Compute the sum of both sides of \eqref{proof_thm_broyden_8} from $t = 0$ to $k - 1$ and use $\sigma_t \leq \epsilon$, \eqref{proof_thm_broyden_2}, and \eqref{proof_thm_broyden_5} to obtain
$$
\sum_{t=0}^{k-1}\frac{\|r_{t+1}\|}{\|r_t\|} \leq (1 + \rho)(1 + \frac{\epsilon}{2})(\sum_{t = 0}^{k - 1}\frac{\sigma_t}{2} + \sum_{t = 0}^{k - 1}\frac{\|(\hat{B}_t - I)\hat{s}_t\|}{\|\hat{s}_t\|}) \leq (1 + \rho)(1 + \frac{\epsilon}{2})(\frac{\epsilon}{2(1 - \rho)} + 2\sqrt{2}\delta q\sqrt{k}).
$$
By leveraging the arithmetic-geometric inequality, we obtain that
\begin{equation}\label{proof_thm_broyden_9}
    \frac{\|r_k\|}{\|r_0\|} = \prod_{t=0}^{k-1}\frac{\|r_{t+1}\|}{\|r_t\|} \leq \left(\frac{\sum_{t=0}^{k-1}\frac{\|r_{t+1}\|}{\|r_t\|}}{k}\right)^k \leq \left(\frac{2\sqrt{2}\delta(1 + \rho)(1 + \frac{\epsilon}{2})q\sqrt{k} + \frac{(1 + \rho)(1 + \frac{\epsilon}{2})\epsilon}{2(1 - \rho)}}{k}\right)^k.
\end{equation}
The proof of \eqref{thm_broyden_3} is complete. Next, we proceed to prove \eqref{thm_broyden_4}. Based on the definition of $J_t$, we have
\begin{align*}
    f(x_t) - f(x_*) & = \nabla{f(x_*)}(x_t - x_*) + \frac{1}{2}(x_t - x_*)^\top J_t(x_t - x_*) = \frac{1}{2}r_t^\top \hat{J}_t r_t,
\end{align*}
where we used $\nabla{f(x_*)} = 0$ and the definitions in \eqref{main_def_4} and \eqref{main_def_5}. By Lemma~\ref{lemma:lemma_3} and $\sigma_t \leq \epsilon$, we have
\begin{equation}\label{proof_thm_broyden_10}
    f(x_0) - f(x_*) = \frac{1}{2}r_0^\top \hat{J}_0 r_0 \geq \frac{1}{2(1 + \frac{\sigma_0}{2})}\|r_0\|^2 \geq \frac{1}{2(1 + \frac{\epsilon}{2})}\|r_0\|^2,
\end{equation}
and
\begin{equation}\label{proof_thm_broyden_11}
    f(x_k) - f(x_*) = \frac{1}{2}r_k^\top \hat{J}_k r_k \leq \frac{1 + \frac{\sigma_k}{2}}{2}\|r_k\|^2 \leq \frac{1 + \frac{\epsilon}{2}}{2}\|r_k\|^2.
\end{equation}
By combining \eqref{proof_thm_broyden_9}, \eqref{proof_thm_broyden_10} and \eqref{proof_thm_broyden_11}, we obtain that 
$$
    \frac{f(x_k) - f(x_*)}{f(x_0) - f(x_*)} \leq \frac{\frac{1 + \frac{\epsilon}{2}}{2}\|r_k\|^2}{\frac{1}{2(1 + \frac{\epsilon}{2})}\|r_0\|^2} \leq \left(1 + \frac{\epsilon}{2}\right)^2\left(\frac{2\sqrt{2}\delta(1 + \rho)(1 + \frac{\epsilon}{2})q\sqrt{k} + \frac{(1 + \rho)(1 + \frac{\epsilon}{2})\epsilon}{2(1 - \rho)}}{k}\right)^{2k},
$$
and the claim in \eqref{thm_broyden_4} holds.
\end{proof}

The above theorem establishes the non-asymptotic superlinear convergence of the Broyden class of quasi-Newton methods. Notice that we use the weighted norm in \eqref{thm_broyden_3} to characterize the convergence rate. Using the fact that $\sqrt{\mu}\|x_t - x_*\| \leq \|r_t\| \leq \sqrt{L}\|x_t - x_*\|, \forall t \geq 0$, the result in \eqref{thm_broyden_3} implies that
\begin{equation}
\frac{\|x_k - x_*\|}{\|x_0 - x_*\|} \leq \sqrt{\frac{L}{\mu}}\left(\frac{2\sqrt{2}\delta(1 + \rho)(1 + \frac{\epsilon}{2})q\sqrt{k} + \frac{(1 + \rho)(1 + \frac{\epsilon}{2})\epsilon}{2(1 - \rho)}}{k}\right)^k, \qquad \forall k \geq 1.
\end{equation}
Next, we use the above theorem to report the results for DFP and BFGS, which are two special cases of the convex Broyden class of quasi-Newton methods. 

\begin{corollary}\label{cor_DFP_BFGS}
Consider the DFP and BFGS methods. Suppose  Assumptions~\ref{ass_str_cvx_smooth} and \ref{ass_Hess_lip} hold and for some $\epsilon, \delta \in (0, \frac{1}{2})$ and $\rho \in (0, 1)$, the initial point $x_0$ and initial Hessian approximation $B_0$ satisfy
\begin{equation}\label{cor_DFP_BFGS_1}
    \frac{M}{\mu^{\frac{3}{2}}}\|\nabla^{2}f(x_*)^\frac{1}{2}(x_0 - x_*)\| \leq \epsilon, \qquad \|\nabla^{2}f(x_*)^{-\frac{1}{2}}\ \! (B_0 - \nabla^{2}f(x_*))\ \!\nabla^{2}f(x_*)^{-\frac{1}{2}}\|_F \leq \delta.
\end{equation}

\begin{itemize}
        \setlength{\parskip}{3pt}
      \item  For the DFP method, if  the tuple $(\epsilon, \delta, \rho)$ satisfies
\begin{equation}\label{cor_DFP_BFGS_2}
\left[\frac{4(2\delta + 1)}{(1 - \epsilon)^2} + \frac{3 + \epsilon}{1 - \epsilon}\right]\frac{\epsilon}{1 - \rho} \leq \delta, \qquad \frac{\epsilon}{2} + 2\delta \leq (1 - 2\delta)\rho \ ,
\end{equation}
then the iterates $\{x_{k}\}_{k=0}^{+\infty}$ generated by the DFP method converge to $x_*$ at a superlinear rate of 
\begin{equation}\label{cor_DFP_BFGS_3}
\frac{\|\nabla^{2}f(x_*)^\frac{1}{2}(x_k - x_*)\|}{\|\nabla^{2}f(x_*)^\frac{1}{2}(x_0 - x_*)\|} \leq \left(\frac{2\sqrt{2}\delta(1 + \rho)(1 + \frac{\epsilon}{2})\sqrt{k} + \frac{(1 + \rho)(1 + \frac{\epsilon}{2})\epsilon}{2(1 - \rho)}}{k}\right)^k, \qquad \forall k \geq 1,
\end{equation}
\begin{equation}\label{cor_DFP_BFGS_4}
\frac{f(x_k) - f(x_*)}{f(x_0) - f(x_*)} \leq \left(1 + \frac{\epsilon}{2}\right)^2\left(\frac{2\sqrt{2}\delta(1 + \rho)(1 + \frac{\epsilon}{2})\sqrt{k} + \frac{(1 + \rho)(1 + \frac{\epsilon}{2})\epsilon}{2(1 - \rho)}}{k}\right)^{2k}, \qquad \forall k \geq 1.
\end{equation}

 \item  For the BFGS method, if the tuple $(\epsilon, \delta, \rho)$ satisfies
\begin{equation}\label{cor_DFP_BFGS_5}
\frac{(3 + \epsilon)\epsilon}{(1 - \epsilon)(1 - \rho)} \leq \delta, \qquad \frac{\epsilon}{2} + 2\delta \leq (1 - 2\delta)\rho \ ,
\end{equation}
then the iterates $\{x_{k}\}_{k=0}^{+\infty}$ generated by the BFGS method converge to $x_*$ at a superlinear rate of 
\begin{equation}\label{cor_DFP_BFGS_6}
\frac{\|\nabla^{2}f(x_*)^\frac{1}{2}(x_k - x_*)\|}{\|\nabla^{2}f(x_*)^\frac{1}{2}(x_0 - x_*)\|} \leq \left(\frac{2\sqrt{2}\delta(1 + \rho)(1 + \frac{\epsilon}{2})\sqrt{\frac{1 + 2\delta}{1 - 2\delta}}\sqrt{k} + \frac{(1 + \rho)(1 + \frac{\epsilon}{2})\epsilon}{2(1 - \rho)}}{k}\right)^k, \qquad \forall k \geq 1,
\end{equation}
\begin{equation}\label{cor_DFP_BFGS_7}
\frac{f(x_k) - f(x_*)}{f(x_0) - f(x_*)} \leq \left(1 + \frac{\epsilon}{2}\right)^2\left(\frac{2\sqrt{2}\delta(1 + \rho)(1 + \frac{\epsilon}{2})\sqrt{\frac{1 + 2\delta}{1 - 2\delta}}\sqrt{k} + \frac{(1 + \rho)(1 + \frac{\epsilon}{2})\epsilon}{2(1 - \rho)}}{k}\right)^{2k}, \qquad \forall k \geq 1.
\end{equation}

\end{itemize}
   
\end{corollary}

\begin{proof}
In Theorem~\ref{thm_broyden}, set  $\phi_k = 1$  to  obtain the results for DFP and set $\phi_k = 0$ to obtain the results for BFGS.
\end{proof}

The results in Corollary~\ref{cor_DFP_BFGS} indicate that, in a local neighborhood of the optimal solution, the  iterates generated by DFP and BFGS converge to the optimal solution at a superlinear rate of $({(C_1\sqrt{k}+C_2)}/{k})^k$, where the constants $C_1$ and $C_2$ are determined by $\rho$, $\epsilon$ and $\delta$. Indeed, as time progresses, the rate behaves as $\mathcal{O}\left((1/{\sqrt{k}})^{k}\right)$. The tuple $(\rho, \epsilon, \delta)$ is independent of the problem parameters $(\mu, L, M, d)$, and the only required condition for the tuple $(\rho, \epsilon, \delta)$ is that they should satisfy \eqref{cor_DFP_BFGS_2} or \eqref{cor_DFP_BFGS_5}. Note that the superlinear rate in \eqref{cor_DFP_BFGS_3} and  \eqref{cor_DFP_BFGS_6} is faster than linear rate of first-order methods as the contraction coefficient  approaches zero at a sublinear rate of $\mathcal{O}({1}/{\sqrt{k}})$. Similarly, in terms of the function value, the superlinear rate shown in \eqref{cor_DFP_BFGS_4} and \eqref{cor_DFP_BFGS_7} behaves as $\mathcal{O}\left((1/k)^{k}\right)$. The result in Corollary~\ref{cor_DFP_BFGS} also shows the existence of a trade-off between the \textit{rate of convergence} and the \textit{neighborhood of superlinear convergence}. We highlight this point in the following remark.

\begin{remark}
There exists a trade-off between the size of the local neighborhood in which DFP or BFGS converges superlinearly and their rate of convergence. To be more precise, by choosing larger values for $\epsilon$ and $\delta$ (as long as they satisfy  \eqref{cor_DFP_BFGS_2} or \eqref{cor_DFP_BFGS_5}), we can increase the size of the region in which quasi-Newton method has a fast superlinear convergence rate, but on the other hand, it will lead to a slower superlinear convergence rate according to the bounds in \eqref{cor_DFP_BFGS_3}, \eqref{cor_DFP_BFGS_4}, \eqref{cor_DFP_BFGS_6} and \eqref{cor_DFP_BFGS_7}. Conversely, by choosing small values for $\epsilon$ and $\delta$, the rate of convergence becomes faster, but the local neighborhood defined in \eqref{cor_DFP_BFGS_1} becomes smaller. 
\end{remark}

The final convergence results of Corollary~\ref{cor_DFP_BFGS} depend on the choice of parameters $(\rho, \epsilon, \delta)$, and it may not be easy to quantify the exact convergence rate at first glance. To better quantify the superlinear convergence rate of DFP and BFGS, in the following corollary, we state the results of Corollary~\ref{cor_DFP_BFGS} for specific choices of $\rho$, $\epsilon$ and $\delta$ which simplifies our expressions. Indeed, one can choose another set of values for these parameters to control the neighborhood and rate of superlinear convergence, as long as they satisfy the conditions in \eqref{cor_DFP_BFGS_2} for DFP and \eqref{cor_DFP_BFGS_5} for BFGS.

\begin{corollary}\label{cor_broyden}
Consider the DFP and BFGS methods and suppose  Assumptions~\ref{ass_str_cvx_smooth} and \ref{ass_Hess_lip} hold. Moreover, suppose the initial point $x_0$ and initial Hessian approximation matrix $B_0$ of DFP satisfy
\begin{equation}\label{cor_broyden_1}
\frac{M}{\mu^{\frac{3}{2}}}\|{\nabla^{2}f(x_*)^\frac{1}{2}}(x_0 - x_*)\| \leq \frac{1}{120}, \qquad \|{\nabla^{2}f(x_*)^{-\frac{1}{2}}}\ \! (B_0 - \nabla^{2}f(x_*))\ \!{\nabla^{2}f(x_*)^{-\frac{1}{2}}}\|_F \leq \frac{1}{7},
\end{equation}
and the initial point $x_0$ and initial Hessian approximation matrix $B_0$ of BFGS satisfy
\begin{equation}\label{cor_broyden_2}
\frac{M}{\mu^{\frac{3}{2}}}\|{\nabla^{2}f(x_*)^\frac{1}{2}}(x_0 - x_*)\| \leq \frac{1}{50}, \qquad \|{\nabla^{2}f(x_*)^{-\frac{1}{2}}}\ \! (B_0 - \nabla^{2}f(x_*))\ \!{\nabla^{2}f(x_*)^{-\frac{1}{2}}}\|_F \leq \frac{1}{7}.
\end{equation}
Then,  the iterates $\{x_k\}_{k=0}^{+\infty}$ generated by the DFP and BFGS methods satisfy
\begin{equation}\label{cor_broyden_3}
\frac{\|\nabla^2{f(x_*)}^{\frac{1}{2}}(x_k - x_*)\|}{\|\nabla^2{f(x_*)}^{\frac{1}{2}}(x_0 - x_*)\|} \leq \left(\frac{1}{k}\right)^{\frac{k}{2}}, \qquad \frac{f(x_k) - f(x_*)}{f(x_0) - f(x_*)} \leq 1.1\left(\frac{1}{k}\right)^{k}, \qquad \forall  k \geq 1.
\end{equation}
\end{corollary}

\begin{proof}
The results for DFP can be shown by setting $\rho = \frac{1}{2}$, $\epsilon = \frac{1}{120}$ and $\delta = \frac{1}{7}$ in Corollary~\ref{cor_DFP_BFGS}. We can check that for those values, the conditions in \eqref{cor_DFP_BFGS_2} are all satisfied. Moreover, the expressions in \eqref{cor_DFP_BFGS_3} and \eqref{cor_DFP_BFGS_4} can be simplified as 
$$
\frac{2\sqrt{2}\delta(1 + \rho)(1 + \frac{\epsilon}{2})\sqrt{k} + \frac{(1 + \rho)(1 + \frac{\epsilon}{2})\epsilon}{2(1 - \rho)}}{k} = \frac{\frac{2\sqrt{2}}{7}(1 + \frac{1}{2})(1 + \frac{1}{240})\sqrt{k} + \frac{(1 + \frac{1}{2})(1 + \frac{1}{240})\frac{1}{120}}{2(1 - \frac{1}{2})}}{k} < \frac{1}{\sqrt{k}},
$$
and $(1 + \epsilon)^2 = (1 + \frac{1}{120})^2 \leq 1.1$. So the claims in \eqref{cor_broyden_3} follow. The results for BFGS can be shown similarly by setting $\rho = \frac{1}{2}$, $\epsilon = \frac{1}{50}$ and $\delta = \frac{1}{7}$ in \eqref{cor_DFP_BFGS_5}, \eqref{cor_DFP_BFGS_6} and \eqref{cor_DFP_BFGS_7}. 
\end{proof}

The results in Corollary~\ref{cor_broyden} show that for some specific choices of $(\epsilon, \delta, \rho)$, the convergence rate of DFP and BFGS is $\left(1/k\right)^{k/2}$, which is asymptotically faster than any linear convergence rate of first-order methods. Moreover, we observe that the neighborhood in which this fast superlinear rate holds is slightly larger for BFGS compared to DFP, i.e., compare the first conditions in \eqref{cor_broyden_1} and \eqref{cor_broyden_2}. This is in consistence with the fact that in practice, BFGS often outperforms DFP.

A major shortcoming of the results in Corollary~\ref{cor_DFP_BFGS} and Corollary~\ref{cor_broyden} is that, in addition to assuming that the initial iterate $x_0$ is sufficiently close to the optimal solution, we also require the initial Hessian approximation error to be sufficiently small. In the following theorem, we resolve this issue by suggesting a practical choice for $B_0$ such that the second assumption in \eqref{cor_broyden_1} and \eqref{cor_broyden_2} can be satisfied under some conditions. To be more precise, we show that if $\|\nabla^2{f(x_*)}^\frac{1}{2}(x_0 - x_*)\|$ is sufficiently small (we formally describe this condition), then by setting $B_0 = \nabla^{2}{f(x_0)}$, the second condition in \eqref{cor_broyden_1} and \eqref{cor_broyden_2} for Hessian approximation is satisfied, and we can achieve the convergence rate in \eqref{cor_broyden_3}.

\begin{theorem}\label{thm_broyden_Hessian}
Consider the DFP and BFGS methods and suppose  Assumptions~\ref{ass_str_cvx_smooth} and \ref{ass_Hess_lip} hold. Moreover, for DFP, suppose  the initial point $x_0$ and initial Hessian approximation $B_0$ satisfy
\begin{equation}\label{thm_broyden_Hessian_1}
    \frac{M}{\mu^{\frac{3}{2}}}\|{\nabla^{2}f(x_*)^\frac{1}{2}}(x_0 - x_*)\| \leq \min\left\{\frac{1}{120}, \frac{1}{7\sqrt{d}}\right\}, \qquad B_0 = \nabla^{2}{f(x_0)},
\end{equation}
and for BFGS, they satisfy
\begin{equation}\label{thm_broyden_Hessian_2}
   \frac{M}{\mu^{\frac{3}{2}}}\|{\nabla^{2}f(x_*)^\frac{1}{2}}(x_0 - x_*)\| \leq \min\left\{\frac{1}{50}, \frac{1}{7\sqrt{d}}\right\}, \qquad B_0 = \nabla^{2}{f(x_0)}.
\end{equation}
Then, the iterates $\{x_k\}_{k=0}^{+\infty}$ generated by the DFP and BFGS methods satisfy
\begin{equation}\label{thm_broyden_Hessian_3}
\frac{\|\nabla^2{f(x_*)}^{\frac{1}{2}}(x_k - x_*)\|}{\|\nabla^2{f(x_*)}^{\frac{1}{2}}(x_0 - x_*)\|} \leq \left(\frac{1}{k}\right)^{\frac{k}{2}}, \qquad \frac{f(x_k) - f(x_*)}{f(x_0) - f(x_*)} \leq 1.1\left(\frac{1}{k}\right)^{k}, \qquad \forall  k \geq 1.
\end{equation}
\end{theorem}

\begin{proof}
First we consider the case of the DFP method. Notice that by \eqref{thm_broyden_Hessian_1}, we obtain
$$
\frac{M}{\mu^{\frac{3}{2}}}\|\nabla^{2}f(x_*)^\frac{1}{2}(x_0 - x_*)\| \leq \frac{1}{120}.
$$
Hence, the first part of \eqref{cor_broyden_1} is satisfied. Moreover, using Assumptions~\ref{ass_str_cvx_smooth} and \ref{ass_Hess_lip}, we have
\begin{align*}
    & \|\nabla^{2}f(x_*)^{-\frac{1}{2}}(\nabla^{2}f(x_0) - \nabla^{2}f(x_*))\nabla^{2}f(x_*)^{-\frac{1}{2}}\|_F\\
    \leq & \sqrt{d}\|\nabla^{2}f(x_*)^{-\frac{1}{2}}(\nabla^{2}f(x_0) - \nabla^{2}f(x_*))\nabla^{2}f(x_*)^{-\frac{1}{2}}\|\\
    \leq & \sqrt{d}\|\nabla^{2}f(x_*)^{-\frac{1}{2}}\|^2\|\nabla^{2}f(x_0) - \nabla^{2}f(x_*)\| \leq \sqrt{d}\frac{M}{\mu}\|x_0 - x_*\|\\
    = & \sqrt{d}\frac{M}{\mu}\|\nabla^{2}f(x_*)^{-\frac{1}{2}}\nabla^{2}f(x_*)^\frac{1}{2}(x_0 - x_*)\| \leq \sqrt{d}\frac{M}{\mu}\|\nabla^{2}f(x_*)^{-\frac{1}{2}}\|\|\nabla^{2}f(x_*)^\frac{1}{2}(x_0 - x_*)\|\\
    \leq & \sqrt{d}\frac{M}{\mu^{\frac{3}{2}}}\|\nabla^{2}f(x_*)^\frac{1}{2}(x_0 - x_*)\| \leq \frac{1}{7}.
\end{align*}
The first inequality holds as $\|A\|_F \leq \sqrt{d}\|A\|$ for any matrix $A \in \mathbb{R}^{d \times d}$, and the last inequality is due to the first part of \eqref{thm_broyden_Hessian_1}. The above bound shows that the second part of the \eqref{cor_broyden_1} is also satisfied, and by Corollary~\ref{cor_broyden} the claim follows. The proof for BFGS is similar to the proof for DFP. It can be derived by following the steps of proof of DFP and exploiting the BFGS results in Corollary~\ref{cor_broyden}.
\end{proof}

According to Theorem~\ref{thm_broyden_Hessian}, if the initial weighted error $\|\nabla^{2}f(x_*)^\frac{1}{2}(x_0 - x_*)\|$ is sufficiently small, then by setting the initial Hessian approximation $B_0$ as the Hessian at the initial point $\nabla^{2}{f(x_0)}$, the iterates will converge superlinearly at a rate of $(1/k)^{k/2}$. More specifically, based on the result in \eqref{lemma_4_4}, it suffices to have $\|\nabla^2{f(x^*)}^{-\frac{1}{2}}\nabla{f(x_0)}\| \leq \mathcal{O}(\mu^{\frac{3}{2}}/(M\sqrt{d}))$  to ensure $\|\nabla^2{f(x^*)}^{\frac{1}{2}}(x_0 - x^*)\| \leq \mathcal{O}(\mu^{\frac{3}{2}}/(M\sqrt{d}))$ as stated in \eqref{thm_broyden_Hessian_1} and \eqref{thm_broyden_Hessian_2}. Hence, this condition is satisfied when $\|\nabla{f(x_0)}\| \leq \mathcal{O}(\mu^2/(M\sqrt{d}))$. This observation implies that, in practice, we can exploit any optimization algorithm to find an initial point $x_0$ such that $\|\nabla{f(x_0)}\| \leq \mathcal{O}(\mu^2/(M\sqrt{d}))$, and once this condition is satisfied, by setting $B_0=\nabla^2 f(x_0)$ we obtain the guaranteed superlinear convergence result. The suggested procedure requires only one evaluation of the Hessian inverse for the initial iterate, and in the rest of the algorithm, the Hessian inverse approximations are updated according to the convex Broyden update in \eqref{Broyden_Hessian_inverse_update}.

\section{Analysis of Self-Concordant Functions}\label{sec:self-concordance}

The results that we have presented so far require three assumptions: (i) the objective function is strongly convex, (ii) its gradient is Lipschitz continuous (iii) and its Hessian is Lipschitz continuous only at the optimal solution. In this section, we extend our theoretical results to a different setting where the objective function is self-concordant.

\begin{assumption}\label{ass_self_concodant}
The objective function $f$ is self-concordant. In other words, it satisfies the following conditions: (i) it is three times continuously differentiable, (ii) $\nabla^2{f(x)} \succ 0$ for all $x \in dom(f)$, and (iii) the Hessian satisfies 
\begin{equation}
\left.\frac{d}{dt}\nabla^2{f(x + ty)}\right|_{t=0} \preceq 2\left(y^\top\nabla^2{f(x)}y\right)^{\frac{1}{2}}\nabla^2{f(x)}, \quad \forall x,y \in dom(f).
\end{equation}
\end{assumption}

The analysis of Newton-type methods for self-concordant functions (see, e.g., \cite{nesterov1989self,nesterov1994interior}) expands the theory of second-order algorithms beyond the classic setting considered in the previous section. This family of functions are of interest as it includes a large set of loss functions that are widely used in machine learning, such as linear functions, convex quadratic functions, and negative logarithm functions. In this section, we extend our results to this class of functions. 

We should mention that the setup considered in this section is neither more general nor more strict than the setup in the previous section. For instance, the function $f(x) = -\log{x}$ is self-concordant and satisfies Assumption~\ref{ass_self_concodant}, but it does not satisfy Assumption~\ref{ass_str_cvx_smooth} or Assumption~\ref{ass_Hess_lip} for any $x > 0$. 
Conversely, the self-concordance assumption is not a necessary condition for the assumption that the Hessian is Lipschitz continuous only at the optimal solution. For instance, the objective function
\begin{equation}
f(x) = 
     \begin{cases}
       7x^2 + 8x + 3 &\quad\text{if} \quad x \in (-\infty, -1)\\
       x^4 + x^2 &\quad\text{if} \quad x \in [-1, 1]\\
       7x^2 - 8x + 3 &\quad\text{if} \quad x \in (1, +\infty)\\
     \end{cases}
\end{equation}
 satisfies the conditions in Assumptions~\ref{ass_str_cvx_smooth} and \ref{ass_Hess_lip}. However, it is not self-concordant, as its third derivative is not continuous. 

Based on these points, the analysis in this section extends our convergence analysis of quasi-Newton methods to a new setting that is not covered by the setup in the previous section. 

We should also mention that in \cite{rodomanov2020greedy,rodomanov2020rates,rodomanov2020ratesnew} for the finite-time analysis of quasi-Newton methods, the authors assume that the objective function is \textit{strongly} self-concordant which forms a subclass of self-concordant functions. Note that a function $f$ is strongly self-concordant when there exists a constant $K \geq 0$ such that for any $x, y, z, w \in dom(f)$, we have
\begin{equation}
\nabla^2{f(y)} -\nabla^2{f(x)} \preceq K\left((y - x)^\top\nabla^2{f(z)}(y - x)\right)^{\frac{1}{2}}\nabla^2{f(w)}.
\end{equation}
In addition, in \cite{rodomanov2020greedy,rodomanov2020rates,rodomanov2020ratesnew} the authors require the objective function to be strongly convex and smooth. Indeed, our considered setting in this section is more general than the setup in these works as we only require the function to be self-concordant. 

Note that the condition $\nabla^2{f(x)} \succ 0$ guarantees that the inner product $s_k^\top y_k$ in quasi-Newton updates is always positive in all iterations, as stated in Section~\ref{sec:QN_methods}. Also by the definition of self-concordance, the function $f(x)$ is always strictly convex.
We start our analysis by stating the following lemma which plays an important role in our
analysis for self-concordant functions.

\begin{lemma}\label{lemma:lemma_9}
Suppose function $f$ satisfies the conditions in  Assumption~\ref{ass_self_concodant}. Further, consider the definition $G := \int_{0}^{1}\nabla^2{f(x + \alpha(y - x))}d\alpha$. 
If $x$ and $y $ are such that $r = \|\nabla^2{f(x)}^{\frac{1}{2}}(y - x)\| < 1$, then 
\begin{equation}\label{lemma_9_1}
    (1 - r)^{2}\nabla^2{f(x)} \preceq \nabla^2{f(y)} \preceq \frac{1}{(1 - r)^{2}}\nabla^2{f(x)},
\end{equation}
\begin{equation}\label{lemma_9_2}
    (1 - r + \frac{r^2}{3})\nabla^2{f(x)} \preceq G \preceq \frac{1}{1 - r}\nabla^2{f(x)}.
\end{equation}
\end{lemma}

\begin{proof}
Check Theorem 4.1.6 and Corollary 4.1.4 of \cite{nesterov2004introductory}.
\end{proof}

The next two lemmas are based on Lemma~\ref{lemma:lemma_9} and are similar to the results in Lemma~\ref{lemma:lemma_3} and \ref{lemma:lemma_4}, except here we prove them for the case that the conditions in Assumption~\ref{ass_self_concodant} are satisfied.

\begin{lemma}\label{lemma:lemma_10}
Recall the definition of $r_k$ in \eqref{main_def_4} and $\hat{J}_k$ in \eqref{main_def_5}. If Assumption~\ref{ass_self_concodant} holds and $\|r_k\| \leq \frac{1}{2}$, then for all $k \geq 0$ we have
\begin{equation}\label{lemma_10_1}
    \frac{1}{1 + 2\|r_k\|}I \preceq \hat{J}_k \preceq (1 + 2\|r_k\|)I.
\end{equation}
\end{lemma}

\begin{proof}
Check Appendix \ref{sec:proof_of_lemma_10}.
\end{proof}

\begin{lemma}\label{lemma:lemma_11}
Recall the definitions in \eqref{main_def_1} - \eqref{main_def_4} and consider the definition $\theta_k := \max\{\|r_k\|, \|r_{k+1}\|\}$. Suppose that for any $k \geq 0$, we have $\theta_k \leq \frac{1}{2}$. If Assumption~\ref{ass_self_concodant} holds, then for all $k \geq 0$ we have
\begin{equation}\label{lemma_11_1}
    \|\hat{y}_k - \hat{s}_k\| \leq 6\theta_k\|\hat{s}_k\|,
\end{equation}
\begin{equation}\label{lemma_11_2}
    (1 - 6\theta_k)\|\hat{s}_k\|^2 \leq \hat{s}_k^\top\hat{y}_k \leq (1 + 6\theta_k)\|\hat{s}_k\|^2,
\end{equation}
\begin{equation}\label{lemma_11_3}
    (1 - 6\theta_k)\|\hat{s}_k\| \leq \|\hat{y}_k\| \leq (1 + 6\theta_k)\|\hat{s}_k\|,
\end{equation}
\begin{equation}\label{lemma_11_4}
    \|\widehat{\nabla{f}}(x_k) - r_k\| \leq 2\|r_k\|^2.
\end{equation}
\end{lemma}

\begin{proof}
Check Appendix \ref{sec:proof_of_lemma_11}.
\end{proof}

By comparing  Lemma~\ref{lemma:lemma_10} and Lemma~\ref{lemma:lemma_11} with Lemma~\ref{lemma:lemma_3} and Lemma~\ref{lemma:lemma_4}, respectively, we observe that the only difference between these results is that we replaced ${\sigma_k}/{2} = ({M}/{(2\mu^{\frac{3}{2}}}))\|r_k\|$ by $2\|r_k\|$ and $\tau_k = \max\{\sigma_k, \sigma_{k+1}\}$ by $6\theta_k = 6\max\{\|r_k\|, \|r_{k+1}\|\}$. Due to this similarity, our results for the self-concordant setting are very similar to the previous case considered in Section~\ref{sec:main_results}. As a result, the superlinear convergence proof in this section is also similar to the one in Section~\ref{sec:main_results}. Next, we directly present the final superlinear convergence rate results for self-concordant functions.

\begin{theorem}\label{thm_broyden_self_concordant}
Consider the convex Broyden class of quasi-Newton methods described in Algorithm~\ref{algo_broyden}. Suppose the objective function $f$ satisfies the conditions in Assumption~\ref{ass_self_concodant}. Moreover, suppose  the initial point $x_0$ and initial Hessian approximation matrix $B_0$ satisfy
\begin{equation}\label{thm_broyden_self_concordant_1}
    \|{\nabla^{2}f(x_*)^\frac{1}{2}}(x_0 - x_*)\| \leq \frac{\epsilon}{6}, \qquad \|{\nabla^{2}f(x_*)^{-\frac{1}{2}}}\ \! (B_0 - \nabla^{2}f(x_*))\ \!{\nabla^{2}f(x_*)^{-\frac{1}{2}}}\|_F \leq \delta,
\end{equation}
where $\epsilon, \delta \in (0, \frac{1}{2})$ such that for some $\rho \in (0, 1)$, they satisfy
\begin{equation}\label{thm_broyden_self_concordant_2}
\max_{k \geq 0}{\left[\phi_k(2\delta + 1)\frac{4}{(1 - \epsilon)^2} + \frac{3 + \epsilon}{1 - \epsilon}\right]}\frac{\epsilon}{1 - \rho} \leq \delta, \qquad \frac{\epsilon}{3} + 2\delta \leq (1 - 2\delta)\rho \ .
\end{equation}
Then the iterates $\{x_{k}\}_{k=0}^{+\infty}$ generated by the convex Broyden class of quasi-Newton methods converge to $x_*$ at a superlinear rate of
\begin{equation}\label{thm_broyden_self_concordant_3}
\frac{\|\nabla^{2}f(x_*)^\frac{1}{2}(x_k - x_*)\|}{\|\nabla^{2}f(x_*)^\frac{1}{2}(x_0 - x_*)\|} \leq \left(\frac{2\sqrt{2}\delta(1 + \rho)(1 + \frac{\epsilon}{3})q\sqrt{k} + \frac{(1 + \rho)(1 + \frac{\epsilon}{3})\epsilon}{3(1 - \rho)}}{k}\right)^k, \qquad \forall k \geq 1,
\end{equation}
\begin{equation}\label{thm_broyden_self_concordant_4}
\frac{f(x_k) - f(x_*)}{f(x_0) - f(x_*)} \leq (1 + \frac{\epsilon}{3})^2\left(\frac{2\sqrt{2}\delta(1 + \rho)(1 + \frac{\epsilon}{3})q\sqrt{k} + \frac{(1 + \rho)(1 + \frac{\epsilon}{3})\epsilon}{3(1 - \rho)}}{k}\right)^{2k}, \qquad \forall k \geq 1,
\end{equation}
where $q = \max_{k \geq 0}{\sqrt{\frac{1}{\phi_k + (1 - \phi_k)\frac{1 - 2\delta}{1 + 2\delta}}}} \in \left[1, \sqrt{\frac{1 + 2\delta}{1 - 2\delta}}\right]$.
\end{theorem}

\begin{proof}
Check Appendix~\ref{sec:proof_of_thm_self_concordant}.
\end{proof}

Similarly, we can set $\phi_k = 1$ or $\phi_k = 0$ to obtain the results for DFP and BFGS, respectively, as stated in Corollary~\ref{cor_DFP_BFGS}. We can also select specific values for $(\epsilon, \delta, \rho)$  to simplify our bounds.

\begin{corollary}\label{cor_broyden_self_concordant}
Consider the DFP and BFGS methods and suppose  Assumption~\ref{ass_self_concodant} holds. Moreover, suppose for the DFP method, the initial point $x_0$ and initial Hessian approximation matrix $B_0$ satisfy
\begin{equation}\label{cor_broyden_self_concordant_1}
\|{\nabla^{2}f(x_*)^\frac{1}{2}}(x_0 - x_*)\| \leq \frac{1}{720}, \qquad \|{\nabla^{2}f(x_*)^{-\frac{1}{2}}}\ \! (B_0 - \nabla^{2}f(x_*))\ \!{\nabla^{2}f(x_*)^{-\frac{1}{2}}}\|_F \leq \frac{1}{7},
\end{equation}
and for the BFGS method, the initial point $x_0$ and initial Hessian approximation matrix $B_0$ satisfy
\begin{equation}\label{cor_broyden_self_concordant_2}
\|{\nabla^{2}f(x_*)^\frac{1}{2}}(x_0 - x_*)\| \leq \frac{1}{300}, \qquad \|{\nabla^{2}f(x_*)^{-\frac{1}{2}}}\ \! (B_0 - \nabla^{2}f(x_*))\ \!{\nabla^{2}f(x_*)^{-\frac{1}{2}}}\|_F \leq \frac{1}{7}.
\end{equation}
Then,  the iterates $\{x_k\}_{k=0}^{+\infty}$ generated by these methods satisfy 
\begin{equation}\label{cor_broyden_self_concordant_3}
\frac{\|\nabla^2{f(x_*)}^{\frac{1}{2}}(x_k - x_*)\|}{\|\nabla^2{f(x_*)}^{\frac{1}{2}}(x_0 - x_*)\|} \leq \left(\frac{1}{k}\right)^{\frac{k}{2}}, \qquad \frac{f(x_k) - f(x_*)}{f(x_0) - f(x_*)} \leq 1.1\left(\frac{1}{k}\right)^{k}, \qquad \forall  k \geq 1.
\end{equation}
\end{corollary}

\begin{proof}
As in the proof of Corollary~\ref{cor_broyden}, we set $\phi_k = 1$, $\rho = \frac{1}{2}$, $\epsilon = \frac{1}{120}$, $\delta = \frac{1}{7}$ for the DFP method and $\phi_k = 0$, $\rho = \frac{1}{2}$, $\epsilon = \frac{1}{50}$, $\delta = \frac{1}{7}$ for the BFGS method in Theorem~\ref{thm_broyden_self_concordant}. Then, the claims follow.
\end{proof}

We can also set the initial Hessian approximation matrix to be $\nabla^2{f(x_0)}$ as in Theorem~\ref{thm_broyden_Hessian} to achieve the same superlinear convergence rate as long as the distance between the initial point $x_0$ and the optimal point $x_*$ is sufficiently small.

\begin{theorem}\label{thm_broyden_Hessian_self_concordant}
Consider the DFP and BFGS methods and suppose  Assumption~\ref{ass_self_concodant} holds. Moreover, suppose for the DFP method, the initial point $x_0$ and initial Hessian approximation matrix $B_0$ satisfy
\begin{equation}\label{thm_broyden_Hessian_self_concordant_1}
    \|{\nabla^{2}f(x_*)^\frac{1}{2}}(x_0 - x_*)\| \leq \min\left\{\frac{1}{720}, \frac{1}{21\sqrt{d}}\right\}, \qquad B_0 = \nabla^{2}{f(x_0)},
\end{equation}
and for the BFGS method, they satisfy
\begin{equation}\label{thm_broyden_Hessian_self_concordant_2}
   \|{\nabla^{2}f(x_*)^\frac{1}{2}}(x_0 - x_*)\| \leq \min\left\{\frac{1}{300}, \frac{1}{21\sqrt{d}}\right\}, \qquad B_0 = \nabla^{2}{f(x_0)}.
\end{equation}
Then, the iterates $\{x_k\}_{k=0}^{+\infty}$ generated by these methods satisfy 
\begin{equation}\label{thm_broyden_Hessian_self_concordant_3}
\frac{\|\nabla^2{f(x_*)}^{\frac{1}{2}}(x_k - x_*)\|}{\|\nabla^2{f(x_*)}^{\frac{1}{2}}(x_0 - x_*)\|} \leq \left(\frac{1}{k}\right)^{\frac{k}{2}}, \qquad \frac{f(x_k) - f(x_*)}{f(x_0) - f(x_*)} \leq 1.1\left(\frac{1}{k}\right)^{k}, \qquad \forall  k \geq 1.
\end{equation}
\end{theorem}

\begin{proof}
First we focus on the DFP method. Notice that by \eqref{thm_broyden_Hessian_self_concordant_1} we have
$$
\|\nabla^{2}f(x_*)^\frac{1}{2}(x_0 - x_*)\| \leq \frac{1}{720}.
$$
Hence, the first condition in  \eqref{cor_broyden_self_concordant_1} is satisfied. Set $x = x_*$ and $y = x_0$ in Lemma~\ref{lemma:lemma_9}. Notice that $\|r_0\| = \|\nabla^2{f(x_*)}^{\frac{1}{2}}(x_0 - x_*)\| \leq \frac{1}{720} < 1$. Hence, using \eqref{lemma_9_1} we obtain that
$$
(1 - \|r_0\|)^2\nabla^2{f(x_*)} \preceq \nabla^2{f(x_0)} \preceq \frac{1}{(1 - \|r_0\|)^2}\nabla^2{f(x_*)}.
$$
Multiply the above expression by $\nabla^2{f(x_*)}^{-\frac{1}{2}}$ from left and right to obtain
$$
(1 - \|r_0\|)^2I \preceq \nabla^2{f(x_*)}^{-\frac{1}{2}}\nabla^2{f(x_0)}\nabla^2{f(x_*)}^{-\frac{1}{2}} \preceq \frac{1}{(1 - \|r_0\|)^2}I,
$$
which implies that
$$
\nabla^{2}f(x_*)^{-\frac{1}{2}}(\nabla^{2}f(x_0) - \nabla^{2}f(x_*))\nabla^{2}f(x_*)^{-\frac{1}{2}} \preceq \left(\frac{1}{(1 - \|r_0\|)^2} - 1\right)I,
$$
$$
\nabla^{2}f(x_*)^{-\frac{1}{2}}(\nabla^{2}f(x_0) - \nabla^{2}f(x_*))\nabla^{2}f(x_*)^{-\frac{1}{2}} \succeq \left((1 - \|r_0\|)^2 - 1\right)I.
$$
The above two inequalities indicate that
\begin{equation}\label{proof_thm_broyden_Hessian_self_concordant}
    \|\nabla^{2}f(x_*)^{-\frac{1}{2}}(\nabla^{2}f(x_0) - \nabla^{2}f(x_*))\nabla^{2}f(x_*)^{-\frac{1}{2}}\| \leq \max\left\{\frac{1}{(1 - \|r_0\|)^2} - 1, 1 - (1 - \|r_0\|)^2\right\}.
\end{equation}
Since $\|r_0\| \in [0, 1)$, we have that
$$
\frac{1}{(1 - \|r_0\|)^2} - 1 = \frac{1 - (1 - \|r_0\|)^2}{(1 - \|r_0\|)^2} \geq 1- (1 - \|r_0\|)^2.
$$
Hence, \eqref{proof_thm_broyden_Hessian_self_concordant} can be simplified as
\begin{align*}
    & \|\nabla^{2}f(x_*)^{-\frac{1}{2}}(\nabla^{2}f(x_0) - \nabla^{2}f(x_*))\nabla^{2}f(x_*)^{-\frac{1}{2}}\|\\
    \leq & \frac{1}{(1 - \|r_0\|)^2} - 1 = \frac{(2 - \|r_0\|)}{(1 - \|r_0\|)^2}\|r_0\| \leq \frac{(2 - \frac{1}{720})}{(1 - \frac{1}{720})^2}\|r_0\| \leq 3\|r_0\|,
\end{align*}
where the second inequality holds due to $\|r_0\| \leq \frac{1}{720}$. Therefore, we can show that
\begin{align*}
    & \|\nabla^{2}f(x_*)^{-\frac{1}{2}}(\nabla^{2}f(x_0) - \nabla^{2}f(x_*))\nabla^{2}f(x_*)^{-\frac{1}{2}}\|_F\\
    \leq & \sqrt{d}\|\nabla^{2}f(x_*)^{-\frac{1}{2}}(\nabla^{2}f(x_0) - \nabla^{2}f(x_*))\nabla^{2}f(x_*)^{-\frac{1}{2}}\| \leq 3\sqrt{d}\|r_0\| \leq \frac{1}{7},
\end{align*}
where the first inequality is true since $\|A\|_F \leq \sqrt{d}\|A\|$ for any matrix $A \in \mathbb{R}^{d \times d}$ and the last inequality is due to the first part of \eqref{thm_broyden_Hessian_self_concordant_1}. Hence, the second condition in  \eqref{cor_broyden_self_concordant_1} is also satisfied. By Corollary~\ref{cor_broyden_self_concordant}, we can conclude that \eqref{thm_broyden_Hessian_self_concordant_3} holds. The proof for BFGS is similar to the proof for DFP. It can be derived by following the steps of proof of DFP and exploiting the BFGS results in Corollary~\ref{cor_broyden_self_concordant}.
\end{proof}

In summary, we established the local convergence rate of the convex Broyden class of quasi-Newton methods  for self-concordant functions. We showed that if the initial distance to the optimal solution is $\|\nabla^{2}f(x_*)^\frac{1}{2}(x_0 - x_*)\| = \mathcal{O}(1)$ and the initial Hessian approximation error is $\|\nabla^{2}{f(x_*)}^{-\frac{1}{2}}(B_0 - \nabla^{2}f(x_*))\nabla^{2}{f(x_*)}^{-\frac{1}{2}}\|_F = \mathcal{O}(1)$, the iterations converge to the optimal solution at a superlinear rate of $\frac{\|\nabla^{2}f(x_*)^\frac{1}{2}(x_k - x_*)\|}{\|\nabla^{2}f(x_*)^\frac{1}{2}(x_0 - x_*)\|} = \mathcal{O}{\left(\frac{1}{\sqrt{k}}\right)^{k}}$ and $\frac{f(x_k) - f(x_*)}{f(x_0) - f(x_*)} = \mathcal{O}{\left(\frac{1}{k}\right)^{k}}$. Moreover, we can achieve the same superlinear rate if the initial error is $\|\nabla^{2}f(x_*)^\frac{1}{2}(x_0 - x_*)\| = \mathcal{O}(\frac{1}{\sqrt{d}})$ and the initial Hessian approximation matrix is $B_0 = \nabla^2{f(x_0)}$.

\section{Discussion}\label{sec:discussions}

In this section, we discuss the strengths and shortcomings of our theoretical results and compare them with concurrent papers \cite{rodomanov2020rates,rodomanov2020ratesnew} on the non-asymptotic superlinear convergence of DFP and BFGS. 

\noindent \textbf{Initial Hessian approximation condition.}
Note that in our main theoretical results, in addition to the fact that the initial iterate $x_0$  has to be close to the optimal solution $x_*$, which is a common condition for local convergence results, we also need the initial Hessian approximation $B_0$ to be close to the Hessian at the optimal solution $\nabla^2 f(x_*)$. At first glance, this might seem restrictive, but as we have shown in Theorem~\ref{thm_broyden_Hessian} and Theorem~\ref{thm_broyden_Hessian_self_concordant}, if we set the initial Hessian approximation to the Hessian at the initial point $\nabla^2 f(x_0)$, this condition is automatically satisfied as long as the initial iterate error $\|x_0-x_*\|$ is sufficiently small. From a practical point of view, this approach is reasonable as quasi-Newton methods and Newton's method outperform first-order methods in a local neighborhood of the optimal solution, and their global linear convergence rate may not be faster than the linear convergence rate of first-order methods. Hence, as suggested in \cite{nesterov2013introductory}, to optimize the overall iteration complexity according to theoretical bounds, one might use first-order methods such as Nesterov's accelerated gradient method to reach a local neighborhood of the optimal solution, and then switch to locally fast methods such as quasi-Newton methods. If this procedure is used, our theoretical results show that by setting $B_0=\nabla^2 f(x_0)$ (and equivalently $H_0=\nabla^2 f(x_0)^{-1}$) for the convex Broyden class of quasi-Newton, the fast superlinear convergence rate of $(1/k)^{k/2}$ can be obtained. 

It is worth noting that, however, in practice algorithms that do not require switching between algorithms or knowledge of problem parameters are more favorable. Due to these reasons, quasi-Newton methods with an Armijo-Wolfe line search are more practical, as they offer an adaptive choice of the steplength with global convergence and avoid specifying typically unknown constants such as the Lipschitz constant of the gradient, Lipschitz constant of the Hessian, and strong convexity parameter.

We should add that the frameworks in \cite{rodomanov2020rates,rodomanov2020ratesnew} require the initial Hessian approximation to be $B_0=LI$, where $I$ is the identity matrix and $L$ is the Lipschitz constant of the gradient. Indeed, satisfying this condition is computationally more affordable than our proposed scheme, as it does not require access to the Hessian or its inverse at the initial iterate $x_0$. However, it still requires a \textit{switching scheme}. To be more precise, it requires to monitor the error of iterates and setting the Hessian approximation as $LI$, once the error $\|x-x_*\|$ is sufficiently small. 
An ideal theoretical guarantee would be compatible with line-search schemes. To be more precise, in both mentioned analyses, we need to monitor the error $\|x-x_*\|$ and reset the Hessian approximation once the error is small. A more comprehensive analysis should be applicable to the case that we follow a line-search approach from the very beginning, and it would automatically  guarantee that once the iterates reach a local neighborhood of the optimal solution, the Hessian approximation for DFP or BFGS satisfies the required conditions for superlinear convergence without requiring to reset the Hessian approximation matrix. That said, the results in this work and \cite{rodomanov2020rates,rodomanov2020ratesnew} are first attempts to study the non-asymptotic behavior of quasi-Newton methods and there is indeed room for improving these results. 

\noindent \textbf{Convergence rate-neighborhood trade-off.}
As mentioned earlier, we observe a trade-off between the radius of the neighborhood in which BFGS and DFP converge superlinearly to the optimal solution and the rate (speed) of superlinear convergence. One important observation here is that for specific choices of $\epsilon$, $\delta$ and $\rho$, the rate of convergence could be independent of the problem dimension $d$, while the neighborhood of the convergence would depend on $d$. Note that by selecting different parameters we could improve the dependency of the neighborhood on $d$, at the cost of achieving a contraction factor that depends on $d$. In this case, the contraction factor may not be always smaller than $1$, and we can only guarantee that after a few iterations it becomes smaller than $1$ and eventually behaves as $1/k$. The results in \cite{rodomanov2020rates,rodomanov2020ratesnew} have a similar structure. For instance, in \cite{rodomanov2020rates}, the authors show that when the initial Newton decrement is smaller than $\frac{\mu^{\frac{5}{2}}}{ML}$, which is independent of the problem dimension, the convergence rate would be of the form $(\frac{dL}{\mu k})^{k/2}$. Hence, to observe the superlinear convergence rate one need to run the BFGS method at least for $d L/\mu$ iterations to ensure the contraction factor is smaller than $1$. A similar conclusion could be made using our results, if we adjust the neighborhood. In our main result, we only report the case that the neighborhood depends on $d$ and the rate is independent of that, since in this case the contraction factor is always smaller than $1$ and the superlinear behavior starts from the first iteration.

\section{Numerical Experiments}\label{sec:experiments}

In this section, we present our numerical experiments and compare the non-asymptotic performance of quasi-Newton methods with Newton's method and the gradient descent algorithm. We further investigate if the convergence rates of quasi-Newton methods are consistent with our theoretical guarantees. In particular, we solve the following logistic regression problem with $l_2$ regularization
\begin{equation}\label{eq_numerical_experiment}
    \min_{x \in \mathbb{R}^d} f(x) = \frac{1}{N}\sum_{i = 1}^{N}\ln{(1 + e^{-y_i z_i^\top x})} + \frac{\mu}{2}\|x\|^2.
\end{equation}
We assume that $\{z_i\}_{i = 1}^{N}$ are the data points and $\{y_i\}_{i = 1}^{N}$ are their corresponding labels where $z_i \in \mathbb{R}^d$ and $y_i \in \{-1, 1\}$ for $1 \leq i \leq N$. Note that the function $f(x)$ in \eqref{eq_numerical_experiment} is strongly convex with parameter $\mu > 0$. We normalize all data points such that $\|z_i\| = 1$ for all $1 \leq i \leq N$. Therefore, the gradient of the function $f(x)$ is Lipschitz continuous with parameter $L = 1+ \mu$. It is also well known that the logistic regression objective function is self-concordant and its Hessian is Lipschitz continuous. In summary, the objective function $f(x)$ defined in \eqref{eq_numerical_experiment} satisfies Assumptions~\ref{ass_str_cvx_smooth}--\ref{ass_Hess_lip} and Assumption~\ref{ass_self_concodant}.

We conduct our experiments on four different datasets: (i) Colon-cancer dataset \cite{colon}, (ii) Covertype dataset \cite{covtype}, (iii) GISETTE handwritten digits classification dataset from the NIPS 2003 feature selection challenge \cite{gisette} and (iv) MNIST dataset of handwritten digits \cite{lecun2010mnist}.\footnote{We use LIBSVM \cite{libsvm} with license: \url{https://www.csie.ntu.edu.tw/~cjlin/libsvm/COPYRIGHT}.} We compare the performance of DFP, BFGS, Newton's method, and gradient descent. We initialize all the algorithms with the same initial point $x_0 = c*\vec{\mathbf{1}}$ where $c > 0$ is a tuned parameter and $\vec{\mathbf{1}} \in \mathbb{R}^d$ is the one vector. We set the initial Hessian inverse approximation matrix as $\nabla^2{f(x_0)}^{-1}$ for the DFP and BFGS methods. The step size is $1$ for DFP, BFGS, and Newton's method. The step size of the gradient descent method is tuned by hand to achieve the best performance on each dataset.

All the parameters (sample size $N$, dimension $d$, initial point parameter $c$ and regularization $\mu$) of these different datasets are provided in Table~\ref{tab_1}. Notice that the initial point parameter $c$ is selected from the set $\mathcal{A} = \{0.001, 0.01, 0.1, 1, 10\}$ to guarantee that the initial point $x_0$ is close enough to the optimal solution $x_*$ so that we can achieve the superlinear convergence rate of DFP and BFGS on each dataset. The regularization parameter $\mu$ is also chosen from the same set $\mathcal{A}$ to obtain the best performance on each dataset.

\begin{table}[t]
  \vspace{1mm}
  \centering
  \begin{tabular}{ |c|c|c|c|c| }
    \hline
    Dataset & $N$ & $d$ & $c$ & $\mu$ \\
    \hline
    \hline
    Colon-cancer & 62 & 2000 & 0.1 & 0.01 \\
    \hline
    Covertype & 581,012 & 54 & 1 & 0.001 \\
    \hline
    GISETTE & 6000 & 5000 & 0.1 & 0.01 \\
    \hline
    MNIST & 11774 & 784 & 0.1 & 0.01 \\
    \hline
  \end{tabular}
  \caption{Sample size $N$, dimension $d$, initial point parameter $c$ and regularization $\mu$ of each dataset.}
  \label{tab_1}
  \vspace{-4mm}
\end{table}

From the theoretical results of Section~\ref{sec:main_result_superlinear_convergence} and Section~\ref{sec:self-concordance}, we expect the iterates $\{x_k\}_{k = 0}^{\infty}$ generated by the DFP method and the BFGS method to satisfy the following superlinear convergence rate

\begin{figure}[t!]
  \centering
  \begin{subfigure}[b]{0.36\linewidth}
    \includegraphics[width=\linewidth]{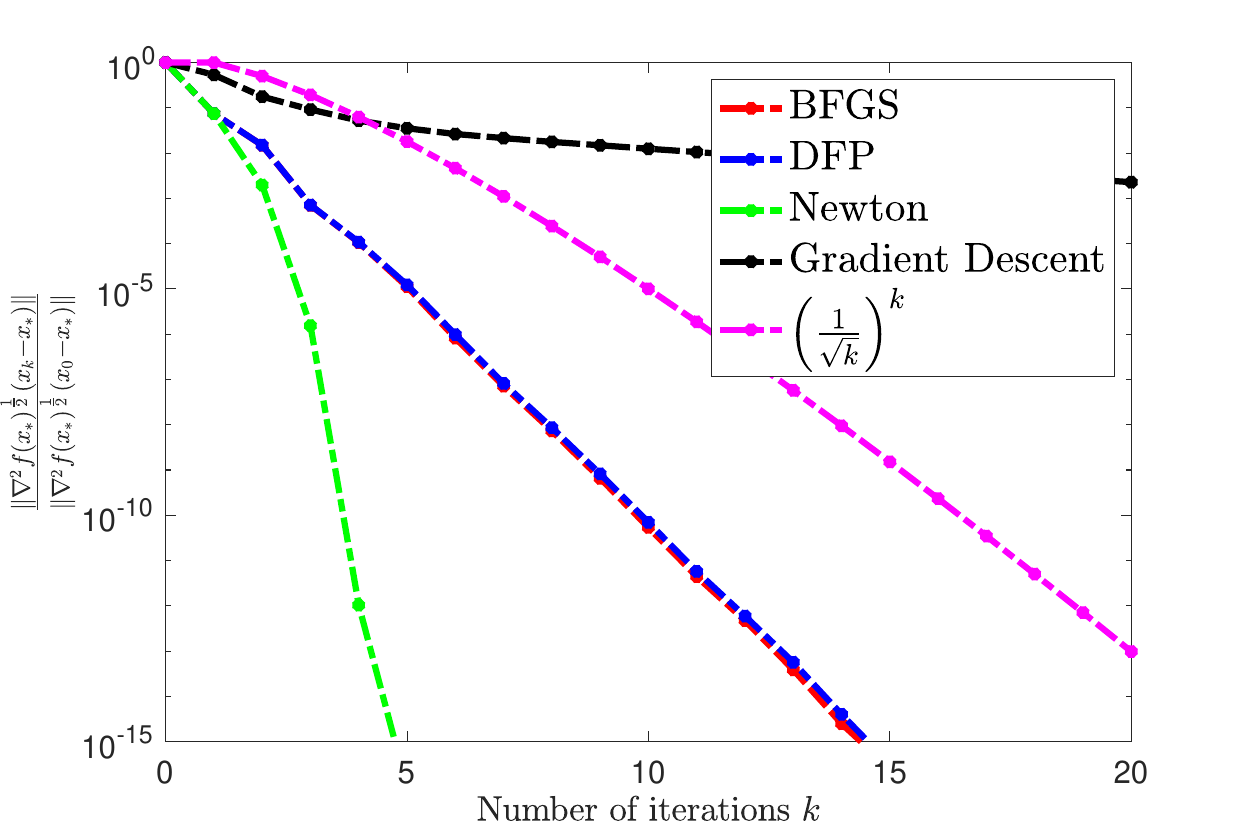}
    \caption{Results of $\frac{\|\nabla^2{f(x_*)}^{{1}/{2}}(x_k - x_*)\|}{\|\nabla^2{f(x_*)}^{{1}/{2}}(x_0 - x_*)\|}$.}
  \end{subfigure}
  \begin{subfigure}[b]{0.36\linewidth}
    \includegraphics[width=\linewidth]{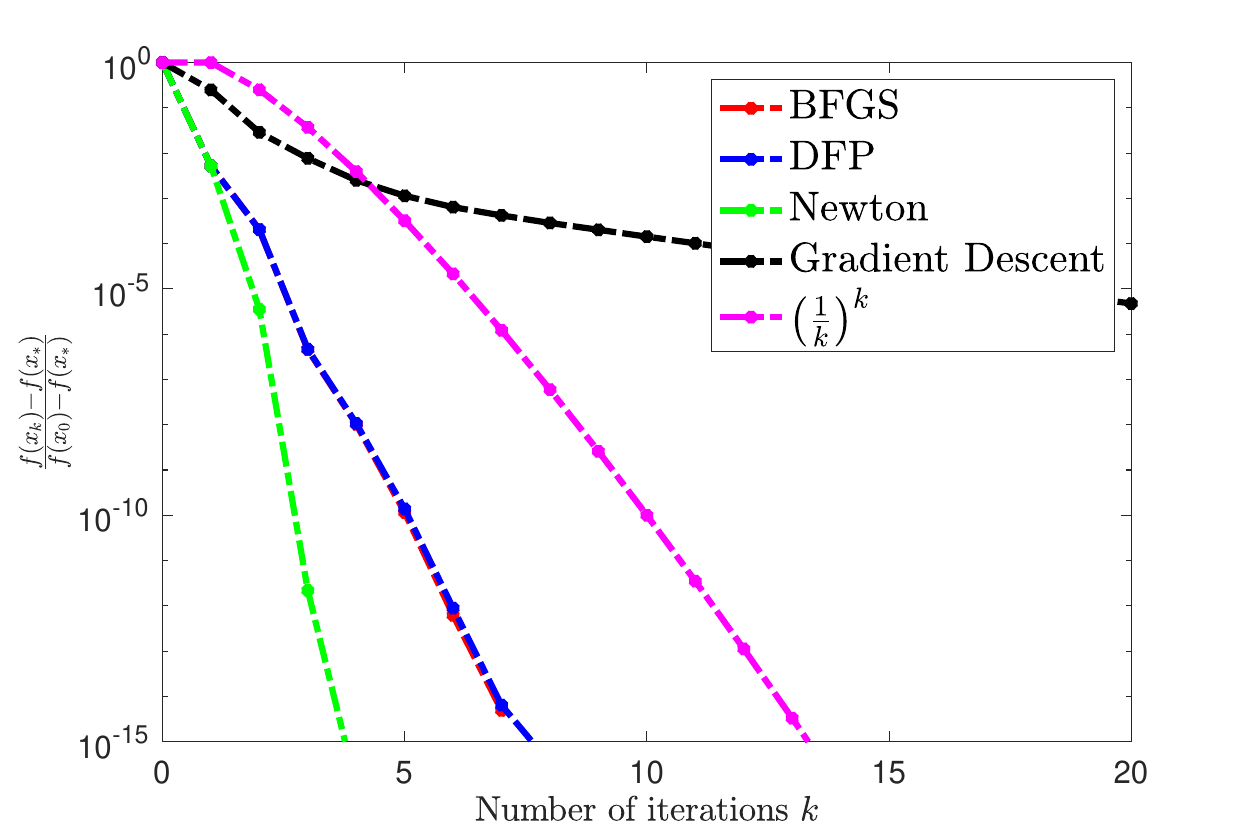}
    \caption{Results of $\frac{f(x_k) - f(x_*)}{f(x_0) - f(x_*)}$.}
  \end{subfigure}
  \captionsetup{labelformat=empty}
  \caption{\textbf{Fig. 1} Convergence rates of logistic regression on the Colon-cancer dataset.}
  \label{fig:f_1}
\end{figure}

\begin{figure}[t!]
  \centering
  \begin{subfigure}[b]{0.36\linewidth}
    \includegraphics[width=\linewidth]{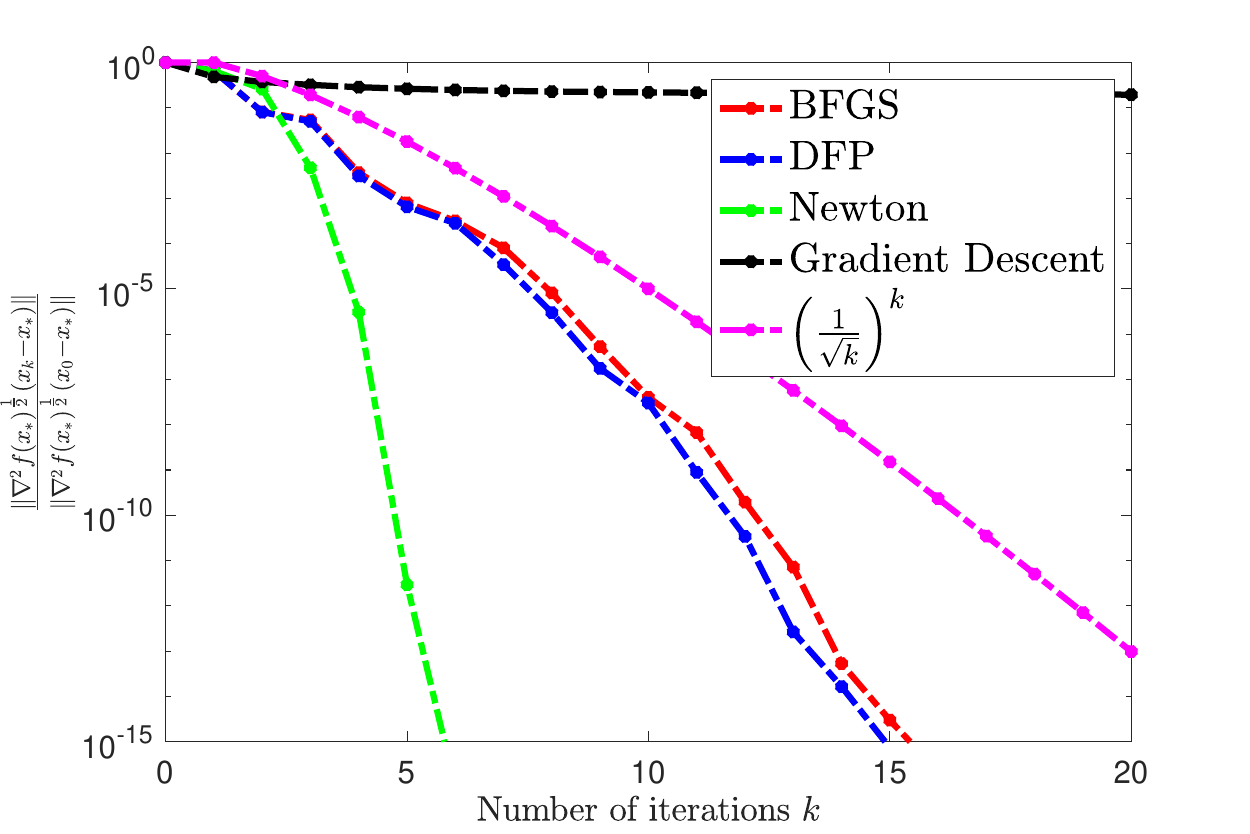}
    \caption{Results of $\frac{\|\nabla^2{f(x_*)}^{{1}/{2}}(x_k - x_*)\|}{\|\nabla^2{f(x_*)}^{{1}/{2}}(x_0 - x_*)\|}$.}
  \end{subfigure}
  \begin{subfigure}[b]{0.36\linewidth}
    \includegraphics[width=\linewidth]{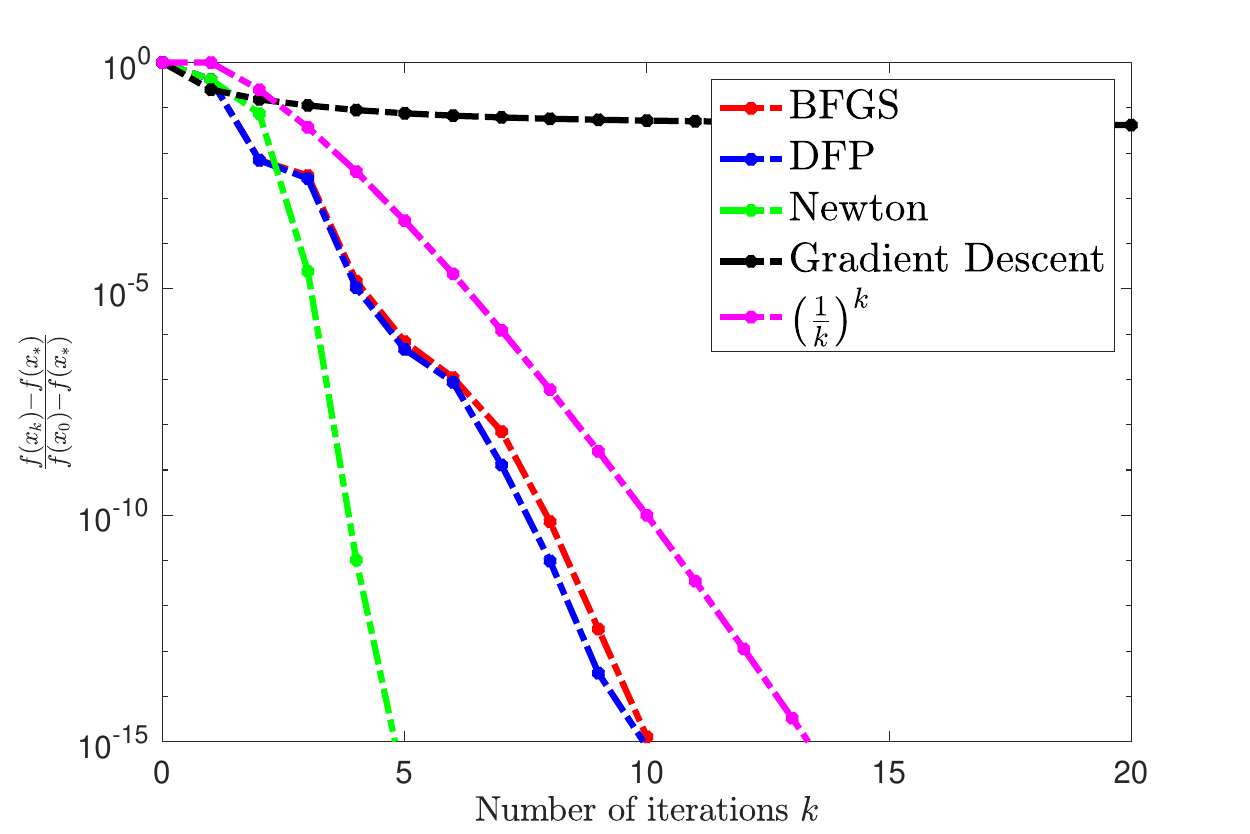}
    \caption{Results of $\frac{f(x_k) - f(x_*)}{f(x_0) - f(x_*)}$.}
  \end{subfigure}
  \captionsetup{labelformat=empty}
  \caption{\textbf{Fig. 2} Convergence rates of logistic regression on the Covertype dataset.}
  \label{fig:f_2}
\end{figure}

\begin{figure}[t!]
  \centering
  \begin{subfigure}[b]{0.36\linewidth}
    \includegraphics[width=\linewidth]{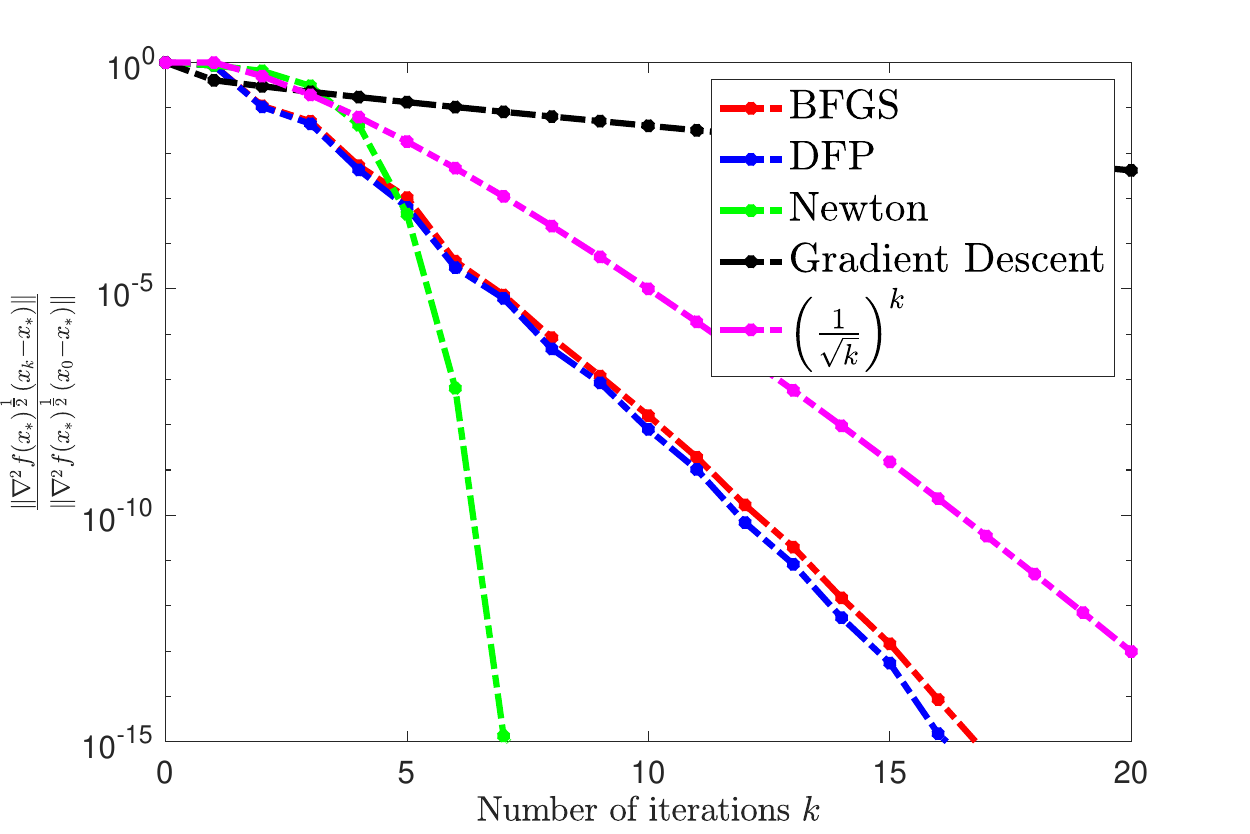}
    \caption{Results of $\frac{\|\nabla^2{f(x_*)}^{{1}/{2}}(x_k - x_*)\|}{\|\nabla^2{f(x_*)}^{{1}/{2}}(x_0 - x_*)\|}$.}
  \end{subfigure}
  \begin{subfigure}[b]{0.36\linewidth}
    \includegraphics[width=\linewidth]{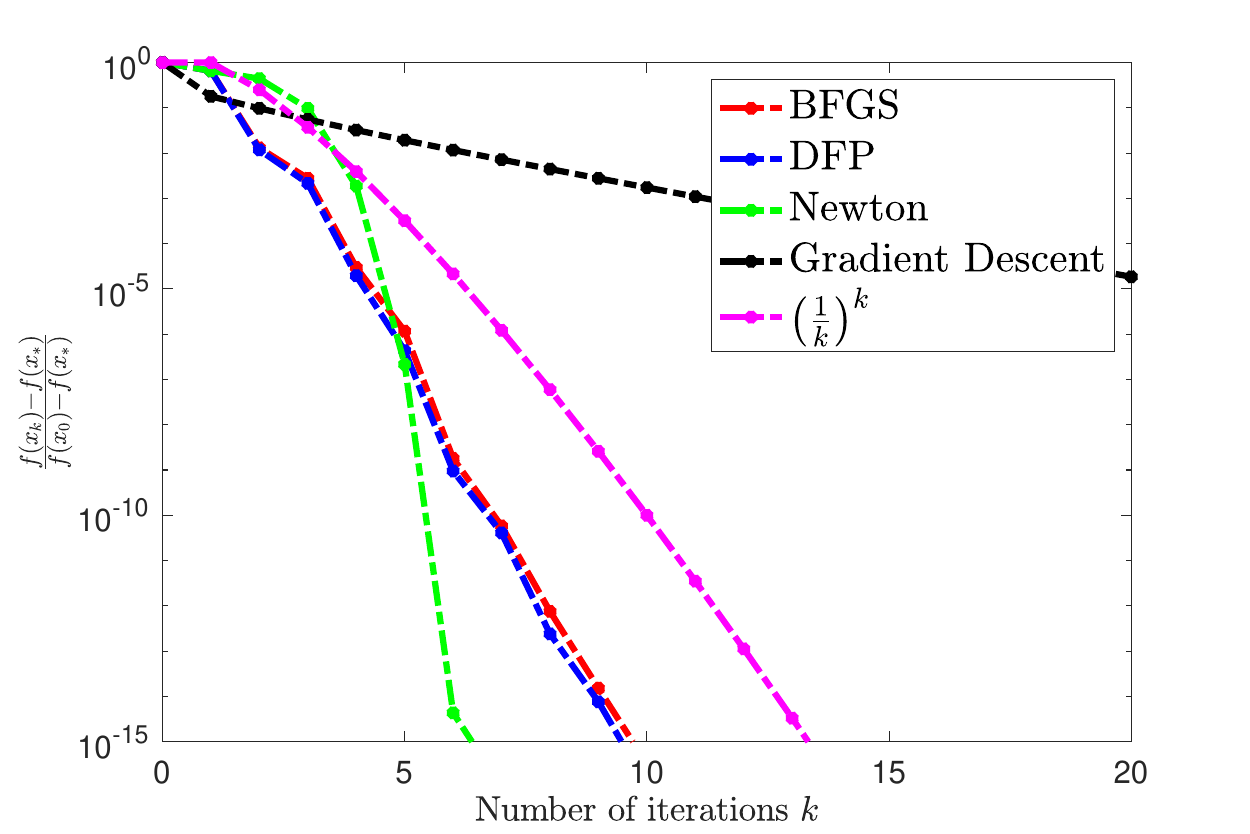}
    \caption{Results of $\frac{f(x_k) - f(x_*)}{f(x_0) - f(x_*)}$.}
  \end{subfigure}
  \captionsetup{labelformat=empty}
  \caption{\textbf{Fig. 3} Convergence rates of logistic regression on the GISETTE dataset.}
  \label{fig:f_3}
\end{figure}

\begin{figure}[t!]
  \centering
  \begin{subfigure}[b]{0.36\linewidth}
    \includegraphics[width=\linewidth]{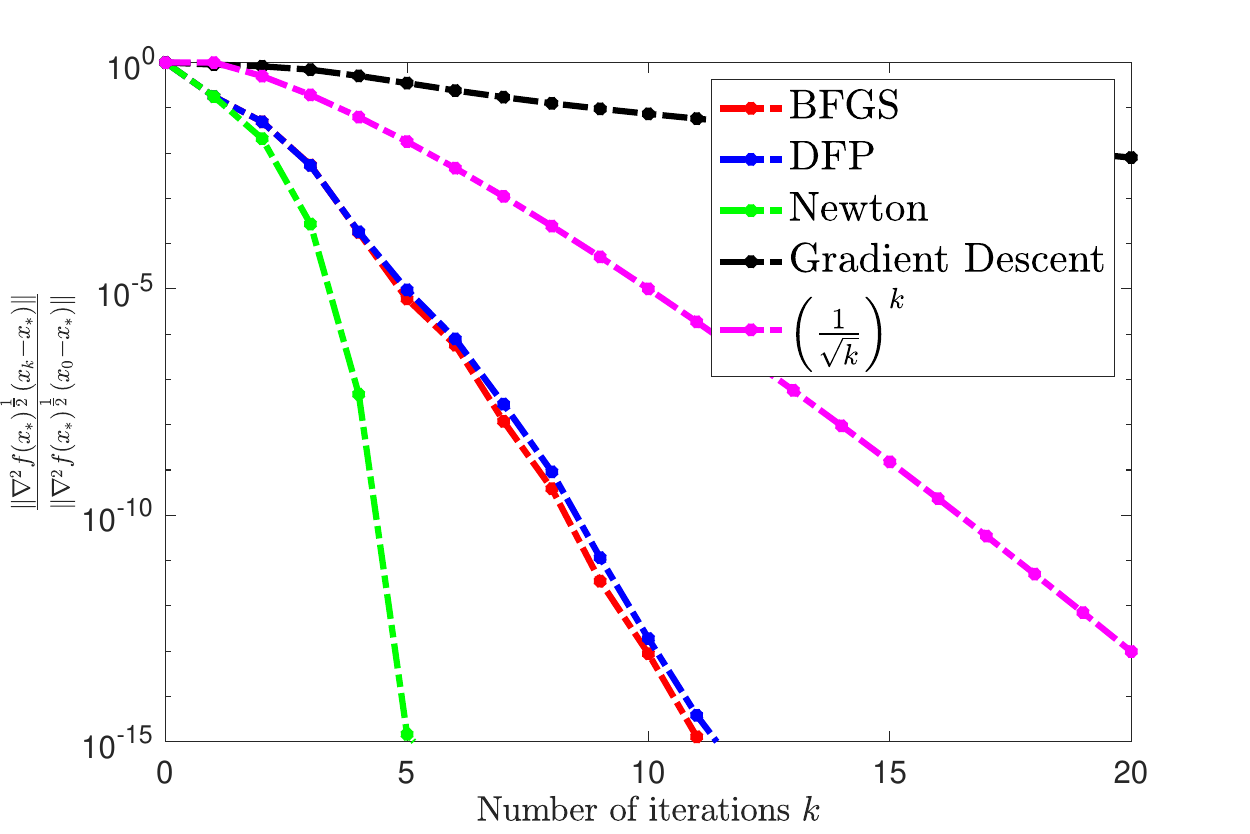}
    \caption{Results of $\frac{\|\nabla^2{f(x_*)}^{{1}/{2}}(x_k - x_*)\|}{\|\nabla^2{f(x_*)}^{{1}/{2}}(x_0 - x_*)\|}$.}
  \end{subfigure}
  \begin{subfigure}[b]{0.36\linewidth}
    \includegraphics[width=\linewidth]{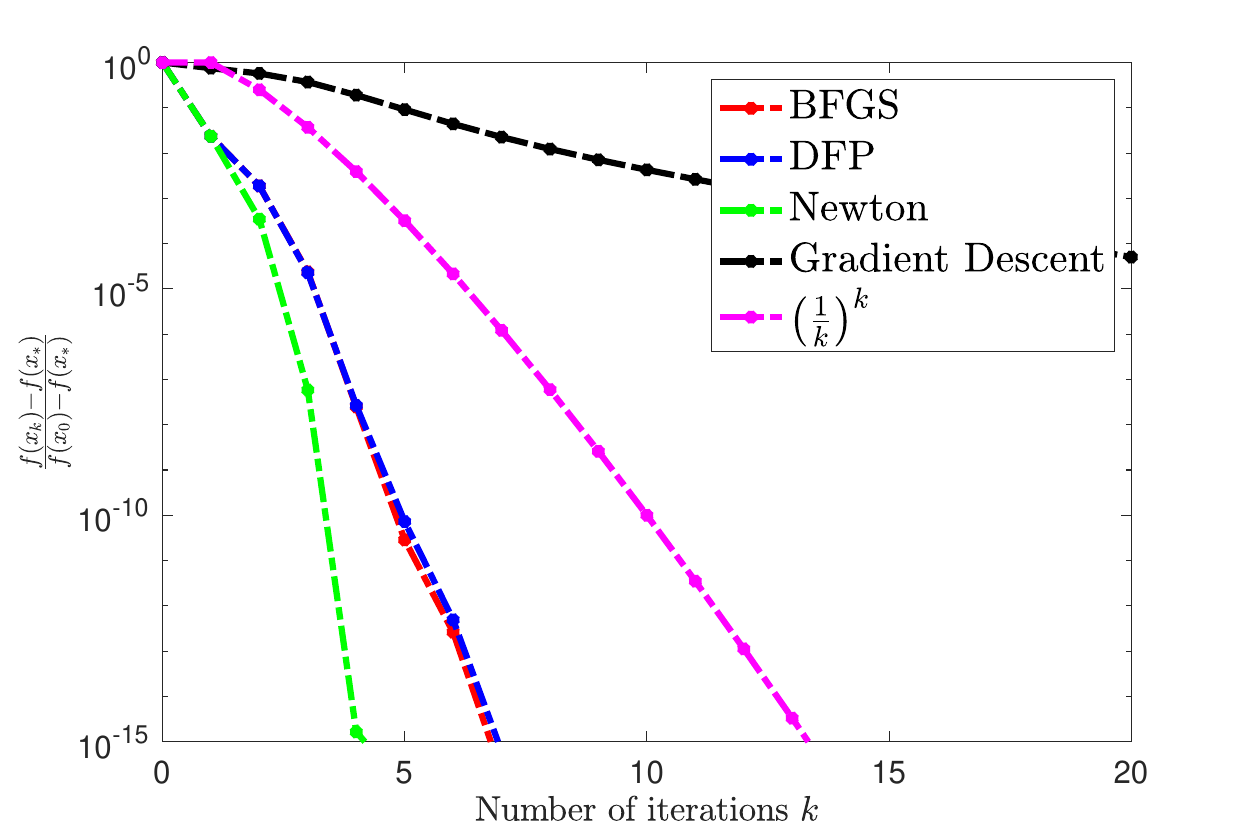}
    \caption{Results of $\frac{f(x_k) - f(x_*)}{f(x_0) - f(x_*)}$.}
  \end{subfigure}
  \captionsetup{labelformat=empty}
  \caption{\textbf{Fig. 4} Convergence rates of logistic regression on the MNIST dataset.}
  \label{fig:f_4}
\end{figure}

$$
\frac{\|\nabla^2{f(x_*)}^{\frac{1}{2}}(x_k - x_*)\|}{\|\nabla^2{f(x_*)}^{\frac{1}{2}}(x_0 - x_*)\|} \leq \left(\frac{1}{\sqrt{k}}\right)^{k}, \qquad \frac{f(x_k) - f(x_*)}{f(x_0) - f(x_*)} \leq 1.1\left(\frac{1}{k}\right)^{k}, \qquad \forall  k \geq 1.
$$
Hence, in our numerical experiments, we compare the convergence rate of $\frac{\|\nabla^2{f(x_*)}^{{1}/{2}}(x_k - x_*)\|}{\|\nabla^2{f(x_*)}^{{1}/{2}}(x_0 - x_*)\|}$ with $(\frac{1}{\sqrt{k}})^{k}$ and the convergence rate of $\frac{f(x_k) - f(x_*)}{f(x_0) - f(x_*)}$ with $(\frac{1}{k})^{k}$ to check the tightness of our theoretical bounds. Our numerical experiments are shown in Figures \ref{fig:f_1}, \ref{fig:f_2}, \ref{fig:f_3} and \ref{fig:f_4} for different datasets. Note that for each problem, we present two plots. The left plot (plot (a)) showcases $\frac{\|\nabla^2{f(x_*)}^{{1}/{2}}(x_k - x_*)\|}{\|\nabla^2{f(x_*)}^{{1}/{2}}(x_0 - x_*)\|}$ for different algorithms as well as our theoretical bound which is $(\frac{1}{\sqrt{k}})^{k}$. In the right plot (plot (b)), we compare $\frac{f(x_k) - f(x_*)}{f(x_0) - f(x_*)}$ for different methods with our theoretical bound which is $(\frac{1}{k})^{k}$. 

We observe that $\frac{\|\nabla^2{f(x_*)}^{{1}/{2}}(x_k - x_*)\|}{\|\nabla^2{f(x_*)}^{{1}/{2}}(x_0 - x_*)\|}$ for the DFP and BFGS methods are bounded above by $(\frac{1}{\sqrt{k}})^{k}$ and $\frac{f(x_k) - f(x_*)}{f(x_0) - f(x_*)}$ for the DFP and BFGS methods are bounded above by $(\frac{1}{k})^{k}$. Therefore, these experimental results confirm our theoretical superlinear convergence rates of quasi-Newton methods.

\section{Conclusion}\label{sec:conclusion}

In this paper, we studied the local convergence rate of the convex Broyden class of quasi-Newton methods which includes the DFP and BFGS methods. We focused on two settings: (i) the objective function is $\mu$-strongly convex, its gradient is  $L$-Lipschitz continuous, and its Hessian is Lipschitz continuous at the optimal solution with parameter $M$, (ii) the objective function is self-concordant. For these two settings we characterized the explicit non-asymptotic superlinear convergence rate of Broyden class of quasi-Newton methods. In particular, for the first setting, we showed that if the initial distance to the optimal solution is $\|\nabla^{2}f(x_*)^\frac{1}{2}(x_0 - x_*)\| = \mathcal{O}(\frac{\mu^{\frac{3}{2}}}{M})$ and the initial Hessian approximation error is $\|{\nabla^{2}f(x_*)^{-\frac{1}{2}}}\ \! (B_0 - \nabla^{2}f(x_*))\ \!{\nabla^{2}f(x_*)^{-\frac{1}{2}}}\|_F = \mathcal{O}(1)$, the iterations generated by the DFP and BFGS methods converge to the optimal solution at a superlinear rate of $\frac{\|\nabla^{2}f(x_*)^\frac{1}{2}(x_k - x_*)\|}{\|\nabla^{2}f(x_*)^\frac{1}{2}(x_0 - x_*)\|} = \mathcal{O}{\left(\frac{1}{\sqrt{k}}\right)^{k}}$ and $\frac{f(x_k) - f(x_*)}{f(x_0) - f(x_*)} = \mathcal{O}{\left(\frac{1}{k}\right)^{k}}$. We further showed that we can achieve the same superlinear convergence rate if the initial error is $\|\nabla^{2}f(x_*)^\frac{1}{2}(x_0 - x_*)\| = \mathcal{O}(\frac{\mu^\frac{3}{2}}{M\sqrt{d}})$ and the initial Hessian approximation matrix is $B_0 = \nabla^2{f(x_0)}$. We proved similar convergence rate results for the second setting where the objective function is self-concordant. 

\section*{Acknowledgment}

This research is supported by NSF Award CCF-2007668. Q. Jin also acknowledges support
from a National Initiative for Modeling and Simulation (NIMS) Graduate Research Fellowship.

\appendix
\section*{Appendix}

\section{Proof of Corollary \ref{corollary}}\label{sec:proof_corollary}

According to the definition $J = \int_{0}^{1}\nabla^2{f(x + t(y - x))}dt$, we have $
\nabla{f(x)} - \nabla{f(y)} = J(x - y)
$. Hence, we can write
\begin{equation}\label{proof_corollary_1}
    \|\nabla{f(x)} - \nabla{f(y)} - \nabla^{2}{f(x_*)}(x - y)\| = \|(J - \nabla^{2}{f(x_*)})(x - y)\| \leq \|J - \nabla^{2}{f(x_*)}\|\|x - y\|.
\end{equation}
Moreover, we can show that
$$
 \|J - \nabla^{2}{f(x_*)}\|  = \left\|\int_{0}^{1}[\nabla^2{f(x + t(y - x))} - \nabla^{2}{f(x_*)}]dt\right\|
     \leq \int_{0}^{1}\|\nabla^2{f(x + t(y - x))} - \nabla^{2}{f(x_*)}\|dt.
$$
By Assumption~\ref{ass_Hess_lip}, we can replace the upper bound in the above expression by the following
\begin{equation}\label{proof_corollary_2}
\begin{split}
    \|J - \nabla^{2}{f(x_*)}\| 
    & \leq M\int_{0}^{1}\|x + t(y - x) - x_*\|dt\\
    & \leq M\left[\int_{0}^{1}(1 - t)\|x - x_*\|dt + \int_{0}^{1}t\|y - x_*\|dt\right] = \frac{M}{2}(\|x - x_*\| + \|y - x_*\|).
\end{split}
\end{equation}
By combining \eqref{proof_corollary_1} and \eqref{proof_corollary_2}, the result in \eqref{corollary_1} follows.

\section{Proof of Lemma \ref{lemma:lemma_1}}\label{sec:proof_of_lemma_1}

Define $P := I - uu^\top$. Since $\|u\| = 1$, we have $P = P^\top$, $P^2 = P$, and $0 \preceq P \preceq I$. These properties imply that
\begin{equation}\label{proof_lemma_1_1}
    \|PAP\|^2_F = \mathrm{Tr}(PAPPA^\top P) = \mathrm{Tr}(PAPA^\top P) = \mathrm{Tr}(PPAPA^\top) = \mathrm{Tr}(PAPA^\top).
\end{equation}
Moreover, for symmetric matrices $X_1$ and $X_2$ that satisfy $X_1 \preceq X_2$ we have $\mathrm{Tr}(X_1 Y) \leq \mathrm{Tr}(X_2 Y)$ when  $Y$ is positive-semidefinite. This result and $0 \preceq P \preceq I$ imply that 
\begin{equation}\label{proof_lemma_1_2}
    \mathrm{Tr}(PAPA^\top) \leq \mathrm{Tr}(APA^\top) = \mathrm{Tr}(AA^\top - Auu^\top A^\top) = \|A\|^2_F - \|Au\|^2.
\end{equation}
By combining the results in  \eqref{proof_lemma_1_1} and \eqref{proof_lemma_1_2}, and considering the definition $P := I - uu^\top$, the claim in  \eqref{lemma_1_1} follows.

\section{Proof of Lemma~\ref{lemma:lemma_2}}\label{sec:proof_of_lemma_2}

Notice that $\mathrm{Tr}(X_1 Y) \leq \mathrm{Tr}(X_2 Y)$ for any symmetric matrices $X_1 \preceq X_2$ and symmetric positive-semidefinite matrix $Y$. Since $A^\top A \preceq \|A\|^2 I$, we obtain that
$$
\|AB\|^2_F = \mathrm{Tr}(B^\top A^\top AB) = \mathrm{Tr}(A^\top ABB^\top) \leq \|A\|^2\mathrm{Tr}(BB^\top) = \|A\|^2\|B\|^2_F,
$$
which leads to the first inequality in \eqref{lemma_2_1}. The second inequality in \eqref{lemma_2_1} follows from the first one, since
$$
\|B^\top AB\|_F \leq \|B^\top A\|\|B\|_F \leq \|A\|\|B\|\|B\|_F.
$$

\section{Proof of Lemma \ref{lemma:lemma_3}}\label{sec:proof_of_lemma_3}

By Assumption~\ref{ass_Hess_lip}, we have that
\begin{equation*}
\|J_k - \nabla^2{f(x_*)}\| = \left\|\int_{0}^{1}\left[\nabla^2{f(x_* + \alpha(x_k - x_*))} - \nabla^2{f(x_*)}\right]d\alpha\right\| \leq \int_{0}^{1}M\alpha\|x_k - x_*\|d\alpha = \frac{M}{2}\|x_k - x_*\|.
\end{equation*}
Hence, we have $J_k - \nabla^2{f(x_*)}\preceq  \frac{M}{2}\|x_k - x_*\|I$. Considering this bound and Assumption~\ref{ass_str_cvx_smooth}, we obtain
\begin{equation}\label{proof_lemma_3_1}
\begin{split}
    J_k - \nabla^2{f(x_*)} &  \preceq \frac{M}{2\mu}\|x_k - x_*\|\nabla^2{f(x_*)} = \frac{M}{2\mu}\|\nabla^2{f(x_*)}^{-\frac{1}{2}}\nabla^2{f(x_*)}^{\frac{1}{2}}(x_k - x_*)\|\nabla^2{f(x_*)}\\
    & \preceq \frac{M}{2\mu}\|\nabla^2{f(x_*)}^{-\frac{1}{2}}\|\|r_k\|\nabla^2{f(x_*)} \preceq \frac{M}{2\mu^{\frac{3}{2}}}\|r_k\|\nabla^2{f(x_*)} = \frac{\sigma_k}{2}\nabla^2{f(x_*)}.
\end{split}
\end{equation}
Similarly, we have that
\begin{equation}\label{proof_lemma_3_2}
    \nabla^2{f(x_*)} - J_k \preceq \|J_k - \nabla^2{f(x_*)}\|I \preceq  \frac{M}{2}\|x_k - x_*\|I \preceq \frac{M}{2\mu}\|x_k - x_*\|J_k \preceq \frac{\sigma_k}{2}J_k.
\end{equation}
Combining \eqref{proof_lemma_3_1} and \eqref{proof_lemma_3_2}, we obtain that
$$
\frac{1}{1 + \frac{\sigma_k}{2}}\nabla^2{f(x_*)} \preceq J_k \preceq (1 + \frac{\sigma_k}{2})\nabla^2{f(x_*)}.
$$
Multiplying both side of the above expression by $\nabla^2{f(x_*)}^{-\frac{1}{2}}$ from left and right leads to the result in \eqref{lemma_3_1}.

\section{Proof of Lemma~\ref{lemma:lemma_4}}\label{sec:proof_of_lemma_4}

By Assumption~\ref{ass_str_cvx_smooth} and Corollary~\ref{corollary}, we have
\begin{equation}\label{proof_lemma_4_1}
\begin{split}
\|\hat{y}_k - \hat{s}_k\| & = \|\nabla^{2}f(x_*)^{-\frac{1}{2}}y_k - \nabla^{2}f(x_*)^{\frac{1}{2}}s_k\| \leq \|\nabla^{2}f(x_*)^{-\frac{1}{2}}\|\|y_k - \nabla^{2}f(x_*)s_k\|\\
& = \|\nabla^{2}f(x_*)^{-\frac{1}{2}}\|\|\nabla{f(x_{k+1})} - \nabla{f(x_k)} - \nabla^{2}f(x_*)(x_{k+1} - x_k)\|\\
& \leq \frac{M}{\mu^{\frac{1}{2}}}\|s_k\|\frac{1}{2}(\|x_{k+1} - x_*\| + \|x_k - x_*\|) \leq \frac{M}{\mu^{\frac{1}{2}}}\|s_k\|\max{\{\|x_{k+1} - x_*\|, \|x_k - x_*\|\}}.
\end{split}
\end{equation}
Notice that
\begin{equation}\label{proof_lemma_4_2}
\|s_k\| = \|\nabla^{2}f(x_*)^{-\frac{1}{2}}\nabla^{2}f(x_*)^{\frac{1}{2}}s_k\| \leq \|\nabla^{2}f(x_*)^{-\frac{1}{2}}\|\|\hat{s}_k\| \leq \frac{1}{\mu^{\frac{1}{2}}}\|\hat{s}_k\|.
\end{equation}
Based on the definition $r_k = \nabla^{2}f(x_*)^{\frac{1}{2}}(x_k - x_*)$, we have $x_k - x_* = \nabla^{2}f(x_*)^{-\frac{1}{2}}r_k$ and hence
\begin{equation}\label{proof_lemma_4_3}
\max{\{\|x_{k+1} - x_*\|, \|x_k - x_*\|\}} \leq \|\nabla^{2}f(x_*)^{-\frac{1}{2}}\|\max{\{\|r_k\|, \|r_{k+1}\|\}} \leq \frac{1}{\mu^{\frac{1}{2}}}\max{\{\|r_k\|, \|r_{k+1}\|\}}.
\end{equation}
Substitute \eqref{proof_lemma_4_2} and \eqref{proof_lemma_4_3} into \eqref{proof_lemma_4_1} and recall the definition in \eqref{main_def_4} to obtain 
$$
\|\hat{y}_k - \hat{s}_k\| \leq \frac{M}{\mu^{\frac{3}{2}}}\max{\{\|r_k\|, \|r_{k+1}\|\}}\|\hat{s}_k\| = \tau_k\|\hat{s}_k\|.
$$
Hence, the proof of the first claim in \eqref{lemma_4_1} is complete. By using the  Cauchy-Schwarz inequality and \eqref{lemma_4_1}, we can write 
$$
|(\hat{y}_k - \hat{s}_k)^\top\hat{s}_k| \leq \|\hat{y}_k - \hat{s}_k\|\|\hat{s}_k\| \leq \tau_k\|\hat{s}_k\|^2.
$$
Therefore, we obtain that $$(1 - \tau_k)\|\hat{s}_k\|^2 \leq \hat{s}_k^\top\hat{y}_k \leq (1 + \tau_k)\|\hat{s}_k\|^2,$$
and the second claim in \eqref{lemma_4_2} holds. Using the reverse triangle inequality and \eqref{lemma_4_1}, we have $|\|\hat{y}_k\| - \|\hat{s}_k\|| \leq \|\hat{y}_k - \hat{s}_k\| \leq \tau_k\|\hat{s}_k\|$. Hence, the third claim in \eqref{lemma_4_3} holds. Finally, to prove the last claim in \eqref{lemma_4_4}, we use Assumption~\ref{ass_str_cvx_smooth} and Corollary~\ref{corollary} to show that
\begin{align*}
\|\widehat{\nabla{f}}(x_k) - r_k\| & = \|\nabla^{2}f(x_*)^{-\frac{1}{2}}\nabla{f(x_k)} - \nabla^{2}f(x_*)^{\frac{1}{2}}(x_k - x_*)\|\\
& \leq \|\nabla^{2}f(x_*)^{-\frac{1}{2}}\|\|\nabla{f(x_k)}  - \nabla{f(x_*)} - \nabla^{2}f(x_*)(x_k - x_*)\|\\
& \leq \frac{M}{2\mu^{\frac{1}{2}}}\|x_k - x_*\|^2 = \frac{M}{2\mu^{\frac{1}{2}}}\|\nabla^{2}f(x_*)^{-\frac{1}{2}}\nabla^{2}f(x_*)^{\frac{1}{2}}(x_k - x_*)\|^2\\
& \leq \frac{M}{2\mu^{\frac{1}{2}}}\|\nabla^{2}f(x_*)^{-\frac{1}{2}}\|^2\|\nabla^{2}f(x_*)^{\frac{1}{2}}(x_k - x_*)\|^2 \leq \frac{M}{2\mu^{\frac{3}{2}}}\|r_k\|^2 = \frac{\sigma_k}{2}\|r_k\|.
\end{align*}

\section{Proof of Lemma \ref{lemma:lemma_10}}\label{sec:proof_of_lemma_10}

Set $x = x_*$ and $y = x_k$ in Lemma~\ref{lemma:lemma_9} and note that $\|r_k\| \leq \frac{1}{2} < 1$. By \eqref{lemma_9_2}, we have
$$
\left(1 - \|r_k\| + \frac{\|r_k\|^2}{3}\right)\nabla^2{f(x_*)} \preceq J_k \preceq \frac{1}{1 - \|r_k\|}\nabla^2{f(x_*)}.
$$
Multiply the above expressions form left and right by   $\nabla^2{f(x_*)}^{-\frac{1}{2}}$ to obtain
\begin{equation}\label{proof_lemma_10_1}
   \left (1 - \|r_k\| + \frac{\|r_k\|^2}{3}\right)I \preceq \hat{J}_k \preceq \frac{1}{1 - \|r_k\|} I.
\end{equation}
Using the fact that $\|r_k\| \leq \frac{1}{2}$, we have
\begin{equation}\label{proof_lemma_10_2}
    1 - \|r_k\| + \frac{\|r_k\|^2}{3} \geq \frac{1}{1 + 2\|r_k\|}
\end{equation}
\begin{equation}\label{proof_lemma_10_3}
    \frac{1}{1 - \|r_k\|} \leq 1 + 2\|r_k\|.
\end{equation}
Replace the lower and upper bounds in  \eqref{proof_lemma_10_1} with the ones in \eqref{proof_lemma_10_2} and \eqref{proof_lemma_10_3}, respectively, to obtain result in \eqref{lemma_10_1}.

\section{Proof of Lemma~\ref{lemma:lemma_11}}\label{sec:proof_of_lemma_11}

We first show that for  $x = x_*$ and $y=x_k + \alpha(x_{k+1} - x_k)$, where $\alpha \in [0, 1]$, the value of $r=\|\nabla^2{f(x)}^{\frac{1}{2}}(y - x)\|$ defined in Lemma~\ref{lemma:lemma_9} is less than 1. To do so, note that 
\begin{align*}
    r & = \|\nabla^2{f(x_*)}^\frac{1}{2}(x_k + \alpha(x_{k+1} - x_k) - x_*)\|\leq \alpha\|r_{k + 1}\| + (1 - \alpha)\|r_k\| \leq  \max\{\|r_k\|, \|r_{k+1}\|\} = \theta_k \leq \frac{1}{2} < 1,
\end{align*}
where $\|r_k\|=\|\nabla^2{f(x_*)}^\frac{1}{2}(x_k-x_*)\|$. Note that in the above simplification we used the assumption that $\theta \leq 1/2$. Now using the result in  \eqref{lemma_9_1} we have
$$
(1 - r)^2\nabla^2{f(x_*)} \preceq \nabla^2{f(x_k + \alpha(x_{k+1} - x_k))} \preceq \frac{1}{(1 - r)^2}\nabla^2{f(x_*)}.
$$
Moreover, since $r \leq \theta_k \in [0, 1)$, we can write
$$
    (1 - \theta_k)^2\nabla^2{f(x_*)} \preceq \nabla^2{f(x_k + \alpha(x_{k+1} - x_k))} \preceq \frac{1}{(1 - \theta_k)^2}\nabla^2{f(x_*)}.
$$
By computing the integral for $\alpha$ from $0$ to $1$ in the above inequality, we get that
$$
(1 - \theta_k)^2\nabla^2{f(x_*)} \preceq G_k \preceq \frac{1}{(1 - \theta_k)^2}\nabla^2{f(x_*)},
$$
where we used the definition $G_k := \int_{0}^{1}\nabla^2{f(x_k + \alpha(x_{k+1} - x_k))}d\alpha$.
Multiplying the above expression from left and right by $\nabla^2{f(x_*)}^{-\frac{1}{2}}$ leads to
$$
(1 - \theta_k)^2I \preceq \hat{G}_k \preceq \frac{1}{(1 - \theta_k)^2}I,
$$
where $\hat{G}_k = \nabla^2{f(x_*)}^{-\frac{1}{2}}G_k\nabla^2{f(x_*)}^{-\frac{1}{2}}$. The above inequality is equivalent to
$$
\left((1 - \theta_k)^2 - 1\right)I \preceq \hat{G}_k - I \preceq \left(\frac{1}{(1 - \theta_k)^2} - 1\right)I,
$$
which indicates that
\begin{equation}\label{proof_of_lemma_11_1}
    \|\hat{G}_k - I\| \leq \max\left\{\frac{1}{(1 - \theta_k)^2} - 1, 1 - (1 - \theta_k)^2\right\}.
\end{equation}
Since $\theta_k \in [0, 1)$, we have that
$$
\frac{1}{(1 - \theta_k)^2} - 1 = \frac{1 - (1 - \theta_k)^2}{(1 - \theta_k)^2} \geq 1 - (1 - \theta_k)^2.
$$
Hence, \eqref{proof_of_lemma_11_1} can be simplified as
\begin{equation}\label{mohem}
\|\hat{G}_k - I\| \leq \frac{1}{(1 - \theta_k)^2} - 1 = \frac{(2 - \theta_k)}{(1 - \theta_k)^2}\theta_k \leq \frac{(2 - \frac{1}{2})}{(1 - \frac{1}{2})^2}\theta_k = 6\theta_k,
\end{equation}
where the second inequality holds due to $\theta_k \leq \frac{1}{2}$. 

Considering the definition $G_k := \int_{0}^{1}\nabla^2{f(x_k + \alpha(x_{k+1} - x_k))}d\alpha$, we have $y_k = G_k s_k$. Using this observation, we have
\begin{align*}
\|\hat{y}_k - \hat{s}_k\| & = \|\nabla^{2}f(x_*)^{-\frac{1}{2}}y_k - \nabla^{2}f(x_*)^{\frac{1}{2}}s_k\|\\
& = \|\nabla^{2}f(x_*)^{-\frac{1}{2}}G_k\nabla^{2}f(x_*)^{-\frac{1}{2}}\nabla^{2}f(x_*)^{\frac{1}{2}}s_k - \nabla^{2}f(x_*)^{\frac{1}{2}}s_k\|\\
& = \|\hat{G}_k\hat{s}_k - \hat{s}_k\| \leq \|\hat{G}_k - I\|\|\hat{s}_k\| \leq 6\theta_k\|\hat{s}_k\|,
\end{align*}
where the last inequality holds due to \eqref{mohem}.
Hence, the proof of the first claim in \eqref{lemma_11_1} is complete.

By using the Cauchy-Schwarz inequality and \eqref{lemma_11_1}, we can write 
$$
|(\hat{y}_k - \hat{s}_k)^\top\hat{s}_k| \leq \|\hat{y}_k - \hat{s}_k\|\|\hat{s}_k\| \leq 6\theta_k\|\hat{s}_k\|^2.
$$
Therefore, we obtain that $$(1 - 6\theta_k)\|\hat{s}_k\|^2 \leq \hat{s}_k^\top\hat{y}_k \leq (1 + 6\theta_k)\|\hat{s}_k\|^2,$$ and the second claim in \eqref{lemma_11_2} holds. Using the reverse triangle inequality and \eqref{lemma_11_1}, we have $|\|\hat{y}_k\| - \|\hat{s}_k\|| \leq \|\hat{y}_k - \hat{s}_k\| \leq 6\theta_k\|\hat{s}_k\|$. Hence, the third claim in \eqref{lemma_11_3} holds. Finally, using Lemma~\ref{lemma:lemma_10} we know that
$$
\|\hat{J}_k - I\| \leq \max\{2\|r_t\|, 1 - \frac{1}{1 + 2\|r_t\|}\} = 2\|r_t\|,
$$
and
\begin{align*}
\|\widehat{\nabla{f}}(x_k) - r_k\| & = \|\nabla^{2}f(x_*)^{-\frac{1}{2}}\nabla{f(x_k)} - \nabla^{2}f(x_*)^{\frac{1}{2}}(x_k - x_*)\|\\
& = \|\nabla^{2}f(x_*)^{-\frac{1}{2}}J_k(x_k - x_*) - \nabla^{2}f(x_*)^{\frac{1}{2}}(x_k - x_*)\|\\
& = \|\hat{J}_k r_k - r_k\| \leq \|\hat{J}_k - I\|\|r_k\| \leq 2\|r_k\|^2.
\end{align*}
Thus, the last claim in \eqref{lemma_11_4} holds.

\section{Proof of Theorem~\ref{thm_broyden_self_concordant}}\label{sec:proof_of_thm_self_concordant}

The proof of this Theorem~\ref{thm_broyden_self_concordant} is very similar to the proof of Theorem~\ref{thm_broyden}. The only difference is that we utilize the Lemma~\ref{lemma:lemma_10} and Lemma~\ref{lemma:lemma_11} instead of Lemma~\ref{lemma:lemma_3} and Lemma~\ref{lemma:lemma_4}. Hence, we need to replace all the terms $\frac{\sigma_k}{2} = \frac{M}{2\mu^{\frac{3}{2}}}\|r_k\|$ by $2\|r_k\|$ and $\tau_k = \max\{\sigma_k, \sigma_{k+1}\}$ by $6\theta_k = 6\max\{\|r_k\|, \|r_{k+1}\|\}$. Here, we only stated the outline of the proof and omit the details to avoid redundancy.

First we present the potential function similar to the \eqref{lemma_7_1} from Lemma~\ref{lemma:lemma_7}. Suppose that for some $\delta > 0$ and some $k \geq 0$, we have $\theta_k = \max\{\|r_k\|, \|r_{k+1}\|\} < \frac{1}{6}$ and $\|\hat{B}_k - I\|_{F} \leq \delta$. Then, the matrix $B_{k+1}$ generated by the convex Broyden class update \eqref{Broyden_Hessian_update} satisfies 
\begin{equation}\label{proof_thm_self_concordant_1}
    \|\hat{B}_{k+1} - I\|_F \leq \|\hat{B}_k - I\|_F - \phi_k\frac{\|(\hat{B}_k - I)\hat{s}_k\|^2}{2\delta\|\hat{s}_k\|^2} - (1 - \phi_k)\frac{\hat{s}_k^\top (\hat{B}_k - I) \hat{B}_k (\hat{B}_k - I) \hat{s}_k}{2\delta \hat{s}_k^\top \hat{B}_k \hat{s}_k} + 6Z_k\theta_k,
\end{equation}
where $Z_k = \phi_k\|\hat{B}_k\|\frac{4}{(1 - 6\theta_k)^2} + \frac{3 + 6\theta_k}{1 - 6\theta_k}$. We also have that
\begin{equation}\label{proof_thm_self_concordant_2}
    \|\hat{B}_{k+1} - I\|_F \leq \|\hat{B}_k - I\|_F + 6Z_k\theta_k.
\end{equation}
The proof of the above conclusion is the same as the proof we presented in Lemmas~\ref{lemma:lemma_5}, \ref{lemma:lemma_6}, and \ref{lemma:lemma_7} except that we use the results of Lemma~\ref{lemma:lemma_11} instead of Lemma~\ref{lemma:lemma_4}. Then, we present the similar linear convergence results like Lemma~\ref{lemma:lemma_8}. Suppose that the objective function $f$ satisfies the conditions in Assumption~\ref{ass_self_concodant}. Moreover, suppose  the initial point $x_0$ and initial Hessian approximation matrix $B_0$ satisfy
\begin{equation}
    \|{\nabla^{2}f(x_*)^\frac{1}{2}}(x_0 - x_*)\| \leq \frac{\epsilon}{6}, \qquad \|{\nabla^{2}f(x_*)^{-\frac{1}{2}}}\ \! (B_0 - \nabla^{2}f(x_*))\ \!{\nabla^{2}f(x_*)^{-\frac{1}{2}}}\|_F \leq \delta,
\end{equation}
where $\epsilon, \delta \in (0, \frac{1}{2})$ such that for some $\rho \in (0, 1)$, they satisfy
\begin{equation}
\max_{k \geq 0}{\left[\phi_k(2\delta + 1)\frac{4}{(1 - \epsilon)^2} + \frac{3 + \epsilon}{1 - \epsilon}\right]}\frac{\epsilon}{1 - \rho} \leq \delta, \qquad \frac{\epsilon}{3} + 2\delta \leq (1 - 2\delta)\rho \ .
\end{equation}
Then, the sequence of iterates $\{x_k\}_{k=0}^{+\infty}$ converges to the optimal solution $x_*$ with
\begin{equation}\label{proof_thm_self_concordant_3}
\|r_{k+1}\| \leq \rho\|r_k\|, \qquad \forall k \geq 0.
\end{equation}
Furthermore, the matrices $\{B_k\}_{k=0}^{+\infty}$ stay in a neighborhood of  $\nabla^{2}{f(x_*)}$ defined as
\begin{equation}\label{proof_thm_self_concordant_4}
\|\hat{B}_{k+1} - I\|_F \leq 2\delta, \qquad \forall k \geq 0.
\end{equation}
Moreover, the norms $\{\|\hat{B}_k\|\}_{k=0}^{+\infty}$ and $\{\|\hat{B}_k^{-1}\|\}_{k=0}^{+\infty}$ are all uniformly bounded above by
\begin{equation}\label{proof_thm_self_concordant_5}
\|\hat{B}_k\| \leq 2\delta + 1, \qquad \|\hat{B}_k^{-1}\| \leq 1 + \rho, \qquad \forall k \geq 0.
\end{equation}
We apply the same induction technique used in the proof of Lemma~\ref{lemma:lemma_8} to prove the above linear convergence results and utilize the potential function in \eqref{proof_thm_self_concordant_2} and Lemma~\ref{lemma:lemma_11}. Finally we can prove the superlinear convergence results of
\begin{equation}
\frac{\|\nabla^{2}f(x_*)^\frac{1}{2}(x_k - x_0)\|}{\|\nabla^{2}f(x_*)^\frac{1}{2}(x_0 - x_*)\|} \leq \left(\frac{2\sqrt{2}\delta(1 + \rho)(1 + \frac{\epsilon}{3})q\sqrt{k} + \frac{(1 + \rho)(1 + \frac{\epsilon}{3})\epsilon}{3(1 - \rho)}}{k}\right)^k, \qquad \forall k \geq 1,
\end{equation}
\begin{equation}
\frac{f(x_k) - f(x_0)}{f(x_0) - f(x_*)} \leq (1 + \frac{\epsilon}{3})^2\left(\frac{2\sqrt{2}\delta(1 + \rho)(1 + \frac{\epsilon}{3})q\sqrt{k} + \frac{(1 + \rho)(1 + \frac{\epsilon}{3})\epsilon}{3(1 - \rho)}}{k}\right)^{2k}, \qquad \forall k \geq 1,
\end{equation}
where $q = \max_{k \geq 0}{\sqrt{\frac{1}{\phi_k + (1 - \phi_k)\frac{1 - 2\delta}{1 + 2\delta}}}} \in \left[1, \sqrt{\frac{1 + 2\delta}{1 - 2\delta}}\right]$. This proof is based on the linear convergence results of \eqref{proof_thm_self_concordant_3}, \eqref{proof_thm_self_concordant_4}, \eqref{proof_thm_self_concordant_5} and is the same as the proof in Theorem~\ref{thm_broyden}, except that here we replace the results of Lemma~\ref{lemma:lemma_3} by the results of Lemma~\ref{lemma:lemma_10}, substitute the results of Lemma~\ref{lemma:lemma_4} with the results of Lemma~\ref{lemma:lemma_11} and utilize the intermediate inequality \eqref{proof_thm_self_concordant_1} instead of \eqref{lemma_7_1}. Notice that all the term $\frac{\epsilon}{2}$ has been replaced with the term $\frac{\epsilon}{3}$ since in this setting, we use the term $2\|r_t\|$ instead of the term $\frac{\sigma_t}{2}$ and $2\|r_t\| \leq 2\|r_0\| \leq 2\frac{\epsilon}{6} = \frac{\epsilon}{3}$.

{{
\bibliographystyle{unsrt}
\bibliography{bmc_article}
}}

\end{document}